\documentclass[12pt]{amsart}
\usepackage{color}
\usepackage{graphicx}
\usepackage{amssymb}
\usepackage{amsthm}
\usepackage{amsmath}
\usepackage{amsfonts}
\usepackage{epstopdf}
\usepackage{url}
\usepackage{hyperref}
\usepackage{fullpage}
\usepackage{float}
\usepackage{verbatim}
\usepackage{enumerate}

\usepackage{tikz}
\usepackage{3dplot}

\theoremstyle{plain}
\newtheorem{thm}{Theorem}[section]
\newtheorem{lem}[thm]{Lemma}
\newtheorem{cor}[thm]{Corollary}
\newtheorem{prop}[thm]{Proposition}

\theoremstyle{definition}
\newtheorem{defn}[thm]{Definition}
\newtheorem{q}[thm]{Question}
\newtheorem{ex}[thm]{Example}
\newtheorem{obs}[thm]{Observation}
\newtheorem{notation}[thm]{Notation}

\numberwithin{equation}{section}
\numberwithin{figure}{section}

\def\w{\omega}

\DeclareMathOperator{\mr}{mr}
\DeclareMathOperator{\mrREAL}{mr^{\mathbb R}}
\DeclareMathOperator{\msr}{mr_+}
\DeclareMathOperator{\msrREAL}{mr^{\mathbb R}_+}
\DeclareMathOperator{\mcr}{mcr}
\DeclareMathOperator{\mcrREAL}{mcr^{\mathbb R}}
\DeclareMathOperator{\mscr}{mcr_{+}}
\DeclareMathOperator{\mscrREAL}{mcr_+^{\mathbb R}}
\DeclareMathOperator{\rank}{rank}
\DeclareMathOperator{\Hqual}{\mathcal{H}}
\DeclareMathOperator{\supp}{supp}
\DeclareMathOperator{\weight}{weight}
\DeclareMathOperator{\Arg}{Arg}

\newcommand{\diag}[1]{\operatorname{diag}\left(#1\right)}
\newcommand{\circulant}[2]{\operatorname{Circ}(#1,#2)}

\renewcommand\implies{\Rightarrow}
\renewcommand\Re{\operatorname{Re}}
\renewcommand\Im{\operatorname{Im}}

\title{The minimum rank problem for circulants}

\author{Louis Deaett}
\email{louis.deaett@quinnipiac.edu}
\address{Department of Mathematics, Quinnipiac University, Hamden, CT 06518}

\author{Seth A.\ Meyer}
\email{seth.meyer@snc.edu}
\address{Department of Mathematics, St.\ Norbert College, De Pere, WI 54115}



\begin{document}

\begin{abstract} 
The {\it minimum rank problem} is to determine for a graph $G$ the smallest rank of a Hermitian (or real symmetric) matrix whose off-diagonal zero-nonzero pattern is that of the adjacency matrix of $G$.  Here $G$ is taken to be a circulant graph, and only circulant matrices are considered.  The resulting graph parameter is termed the {\it minimum circulant rank} of the graph.  This value is determined for every circulant graph in which a vertex neighborhood forms a consecutive set, and in this case is shown to coincide with the usual minimum rank.
Under the additional restriction to positive semidefinite matrices, the resulting parameter is shown to be equal to the smallest number of dimensions in which the graph has an orthogonal representation with a certain symmetry property, and also to the smallest number of terms appearing among a certain family of polynomials determined by the graph.  This value is then determined when the number of vertices is prime.  The analogous parameter over $\mathbb R$ is also investigated.
\end{abstract}

\maketitle

\section{Introduction}\label{sec:intro}

The location of the off-diagonal nonzero entries of a Hermitian or real symmetric matrix can naturally be specified by a graph.  More formally, we have the following.

\begin{defn}\label{def:graph_of_matrix}
Let $A$ be an $n\times n$ Hermitian matrix and $G$ be a simple graph on
$n$ vertices, say with $V(G)=\{v_1,v_2,\ldots,v_n\}$.  We say that $G$ is the {\it graph} of $A$ if it is the case that $\{v_i,v_j\} \in E(G)$ if and only if $a_{ij} \not= 0$, for all $i,j\in\{1,2,\ldots,n\}$ with $i\not=j$.
\end{defn}

A problem of interest in combinatorial matrix theory is to
determine particular ways in which the graph of a matrix constrains its rank.
Because the diagonal entries of the matrix play no role in Definition \ref{def:graph_of_matrix}, every graph allows a diagonally dominant matrix, so the question of how large the rank may be is not interesting.
On the other hand, to determine the smallest rank among all matrices with a given graph is an interesting problem, known as the {\it minimum rank problem} for graphs.  More formally, the problem is to determine the value of the graph parameter defined as follows.

\begin{notation}\label{not:qual_classes}
Let $G$ be a graph.  We write $\Hqual(G)$ for the set of all complex Hermitian matrices with graph $G$.
\end{notation}

\begin{defn}\label{def:mr}
Let $G$ be a graph.  The {\it minimum rank} of $G$ is
\[ \mr(G) = \min\{\rank(A) : A \in \Hqual(G) \}. \]
\end{defn}

The present work focuses on the case in which $G$ is a circulant graph.  We may then consider the smallest rank among all Hermitian (or real symmetric) circulant matrices whose off-diagonal nonzero entries occur according to the edges of $G$.

A question that naturally arises is:  When $G$ is a circulant {\em graph}, under what conditions is the smallest rank among all Hermitian (or real symmetric) matrices with graph $G$ attained by a circulant  {\em matrix?}
In Section \ref{sec:mcr_for_families} we show that this in fact does occur for at least one broad class of circulant graphs,
namely those in which each vertex neighborhood comprises a consecutive set of vertices.
We also give examples of circulants for which this does not occur, however the problem of providing a complete characterization of such circulants remains open.

We also investigate the problem in the positive semidefinite setting.  First, in Section \ref{sec:orthogonal_representations_and_symmetry}, we show that the problem of determining the smallest rank among all positive semidefinite circulant matrices with a given graph is equivalent to determining the smallest number of dimensions admitting an orthogonal representation for the graph with a specific symmetry property.  Then, in Section \ref{sec:polynomial_connection}, this problem in turn is shown to be equivalent to determining the smallest number of terms in a real polynomial with nonnegative coefficients whose zeros intersect a precise subset, determined by the graph, of the complex roots of unity.  
In Section \ref{sec:mcr_for_families}, this value is determined for two broad classes of circulants.
Finally, Section \ref{sec:results_over_the_reals} develops analogous results over $\mathbb R$.

\section{Preliminaries}\label{prelims}

We begin with some fundamental definitions.  In particular, we need to set out what is meant by a {\it circulant}, in both the sense of a graph and of a matrix, and then establish appropriate connections between the two notions.

\subsection{Circulant graphs}

Intuitively, $G$ is a circulant graph precisely when its vertices may be arranged around a circle such that the presence of an edge between any two vertices is determined entirely by the distance ``around the circle'' from one to the other.  More precisely, we have the following definition.

\begin{defn}\label{def:circulant_graph}
A graph $G$ on $n$ vertices is said to be a {\it circulant graph} if its vertices may be labeled as $v_0,v_1,\ldots,v_{n-1}$ such that there exists a set $S \subseteq \{1,2,\ldots,n-1\}$ with
\[ \{v_i, v_j\} \in E(G) \Longleftrightarrow i-j \equiv k \bmod n \text{ for some } k \in S. \]
\end{defn}

When $v \in V(G)$, we write $N(v)$ for the {\it neighborhood} of $v$, i.e., the set of all vertices adjacent to $v$.  Note that when $G$ is a circulant graph, its entire edge set is determined by the neighborhood of any single vertex.  Hence, it is convenient to specify a circulant graph by giving its number of vertices together with the neighborhood of just one vertex.

\begin{notation}
Let $n$ be a positive integer and $S \subseteq \{1,2,\ldots,n-1\}$ be closed under taking the additive inverse modulo $n$.  By $\circulant{n}{S}$ we denote the unique circulant graph on vertex set $\{0,1,\ldots,n-1\}$ such that $N(0)=S$.  For the sake of convenience, when writing the elements of a specific set $S$, we do not restrict ourselves to integers in $\{1,2,\ldots,n-1\}$, but may instead write other integers that represent the same residues modulo $n$.
\end{notation}

Note that, by definition, $\circulant{n}{S}$ is necessarily regular  of degree $|S|$.

\begin{defn}\label{def:consecutive_circulant}
A graph $G$ is a {\it consecutive circulant} if there exist integers $n$ and $k$ with $1 \le k \le \lfloor n/2 \rfloor$ such that  $G=\circulant{n}{S}$ for $S=\{\pm 1, \pm 2, \ldots, \pm k\}$.
\end{defn}

Note that Definition \ref{def:consecutive_circulant} precludes a graph with no edges; we do not consider the empty graph to be a consecutive circulant.

\begin{ex}
The circulant graphs $\circulant{10}{\{\pm 1,\pm 2, \pm 3\}}$ and $\circulant{6}{\{\pm 2,3\}}$
are shown in Figure \ref{fig:circulant_graph_examples}.  The former is a consecutive circulant.
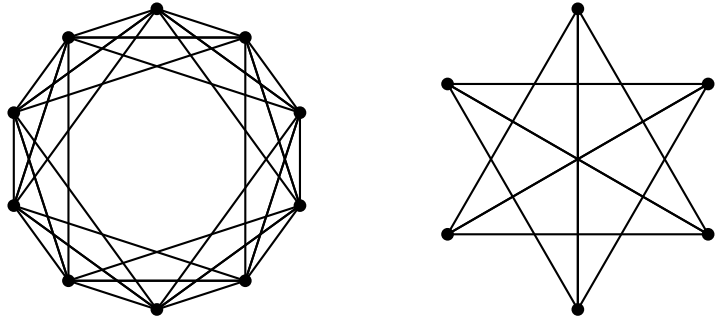
\begin{figure}
\begin{center}
\begin{tabular}{ccc}
\begin{tikzpicture}
\def\n{10};
\foreach \x in {1,...,\n}{
	\pgfmathsetmacro\xcoord{2*(sin((360/(\n)*(\x))))};
	\pgfmathsetmacro\ycoord{2*(cos((360/(\n)*(\x))))};
    \filldraw[thick] (\xcoord,\ycoord) circle (2pt);
    \foreach \e in {1,2,8,7}{
        \pgfmathsetmacro\xecoord{2*(sin((360/(\n)*(\x+\e))))};
        \pgfmathsetmacro\yecoord{2*(cos((360/(\n)*(\x+\e))))};
        \draw[thick] (\xcoord,\ycoord) -- (\xecoord,\yecoord);};};
\end{tikzpicture}
& \hspace{6ex} &
\begin{tikzpicture}
\def\n{6};
\foreach \x in {1,...,\n}{
	\pgfmathsetmacro\xcoord{2*(sin((360/(\n)*(\x))))};
	\pgfmathsetmacro\ycoord{2*(cos((360/(\n)*(\x))))};
    \filldraw[thick] (\xcoord,\ycoord) circle (2pt);
    \foreach \e in {2,3}{
        \pgfmathsetmacro\xecoord{2*(sin((360/(\n)*(\x+\e))))};
        \pgfmathsetmacro\yecoord{2*(cos((360/(\n)*(\x+\e))))};
        \draw[thick] (\xcoord,\ycoord) -- (\xecoord,\yecoord);};};
\end{tikzpicture}
\end{tabular}
\end{center}
\caption{The circulant graphs $\circulant{10}{\{\pm 1,\pm 2,\pm 3\}}$ and $\circulant{6}{\{\pm 2,3\}}$.}
\label{fig:circulant_graph_examples}
\end{figure}
\end{ex}

\subsection{Circulant matrices}
In what follows, the rows and columns of every matrix, as well as the coordinates of every vector, are taken to be zero-indexed.  That is, the coordinates of each $n$-dimensional vector, as well as the rows and the columns of each $n \times n$ matrix will be indexed by the integers $0,1,\ldots,n-1$.  We write $A_{ij}$ for the entry residing in row $i$ and column $j$ of matrix $A$.

Accordingly, given subsets $S$ and $T$ of the set $\{0,1,\ldots,n-1\}$, we write $A[S,T]$ for the submatrix of $A$ that lies on the rows of $A$ with indices in $S$ and the columns with indices in $T$.  When $S=T$, we denote this submatrix simply by $A[S]$.

\begin{defn}
An $n \times n$ matrix $A$ is a {\it circulant matrix} if
there exist values for $b_0,b_1,\ldots,b_{n-1}$
such that
\[ A = \begin{bmatrix}
b_0 & b_1 & b_2 &\cdots & b_{n-2} & b_{n-1} \\
b_{n-1} & b_0 & b_1 & \cdots & b_{n-3} & b_{n-2} \\
b_{n-2} & b_{n-1} & b_0 & \cdots & b_{n-4} & b_{n-3} \\
\vdots & \vdots & \vdots & \ddots & \vdots & \vdots \\
b_2 & b_3 & b_4 & \cdots & b_0 & b_1 \\
b_1 & b_2 & b_3 &\cdots & b_{n-1} & b_0
\end{bmatrix}. \]
\end{defn}

Note that $A$ is a circulant matrix if and only if the value of $A_{ij}$ is a function of the residue of $i-j$ modulo $n$.  For a general reference on circulant matrices and their properties, see the monograph \cite{Davis} or the more recent \cite{KS2012}.

\begin{notation}
In any context in which $n$ is understood to be a positive integer, we write $\w$ for the complex $n$th root of unity $e^{2\pi i/n}$.
\end{notation}

\begin{defn}\label{def:fourier_matrix}
For every positive integer $n$, the $n\times n$ {\it Fourier matrix} $F_n$ is defined by $(F_n)_{ij} = \frac{1}{\sqrt{n}}\w^{ij}$.
At times we wish to neglect the scaling factor of $1/\sqrt n$, and so for convenience of notation we let $\hat F_n = \sqrt nF_n$.
Explicitly, we have
\[ F_n = {\textstyle \frac 1{\sqrt n}}\hat F_n = 
\frac 1{\sqrt n}
\begin{bmatrix}
1 & 1 & 1 & \cdots & 1 \\
1 & \w & \w^2 & \cdots & \w^{n-1} \\
1 & \w^2 & \w^4 & \cdots & \w^{2(n-1)} \\
\vdots & \vdots & \vdots & \ddots & \vdots \\
1 & \w^{n-1} & \w^{2(n-1)} & \cdots & \w^{(n-1)(n-1)}
\end{bmatrix}. \]
\end{defn}

Note that $\hat F_n$ is the Vandermonde matrix of the complex $n$th roots of unity.  Note also that $F_n^\ast = F_n^{-1}$, so that $F_n$ is unitary.  The significance of the Fourier matrix in the context of circulant matrices stems from the fact that $F_n$ diagonalizes every $n\times n$ circulant matrix.

\begin{thm}[see, e.g., {\cite[Theorem 3.2.2]{Davis}}]\label{thm:circulants_diagonalized_by_Fourier}
If $A$ is an $n\times n$ circulant matrix, then $F_n^\ast A F_n$ is diagonal.
\end{thm}

\subsection{Minimum circulant rank}
As established above, the term ``circulant'' has a distinct meaning as applied to a matrix as opposed to a graph.  When that distinction is clear from the context, however, we will sometimes refer to an object simply as a ``circulant'' for the sake of brevity.

The classical minimum rank problem seeks the smallest rank among all Hermitian (or real symmetric) matrices with the zero-nonzero pattern specified by a given graph.
We wish to consider this problem with its scope restricted to the circulant matrices.
Hence, for a given circulant graph, we seek the smallest rank of a Hermitian or real symmetric circulant matrix whose zero-nonzero pattern is that given by the graph.
That is, we wish to study the following graph parameter analogous to that set out in Definition \ref{def:mr}.

\begin{defn}
Let $G$ be a circulant graph.  The {\it minimum circulant rank} of $G$ is
\[ \mcr(G) = \min\{\rank(A) : A \in \Hqual(G) \text{ and $A$ is a circulant matrix} \}. \]
\end{defn}

The following easy observation shows that $\mcr(G)$ is well-defined.

\begin{obs}
A graph $G$ is a {\it circulant graph} if and only if the adjacency matrix of $G$ can be taken to be a circulant matrix.
\end{obs}

Occasionally, we will consider some previously-defined minimum rank parameter over a specific field other than $\mathbb C$.  To be precise, we establish the following notation.

\begin{notation}
When $K$ denotes a field, we write $K$ as a superscript to denote an analogous minimum rank parameter defined such that the matrices in question are taken to be symmetric with entries in $K$.  For instance, $\mcrREAL(G)$ denotes the minimum rank over all real symmetric circulant matrices with graph $G$.
\end{notation}

We refer to the problem of determining the value of $\mcr(G)$ for a graph $G$ as the {\it circulant minimum rank problem}.
The following question thus arises naturally.

\begin{q}\label{q:mr_equal_mcr_when}
What conditions on a circulant graph $G$ are sufficient to ensure that $\mr(G) = \mcr(G)$?
\end{q}

The extent to which a simple or complete answer to Question \ref{q:mr_equal_mcr_when} exists is certainly not clear.  Nevertheless, Theorem \ref{thm:consCircC} provides a partial result in this direction, showing that the equality of Question \ref{q:mr_equal_mcr_when} does in fact hold for at least one natural family of circulant graphs.

\subsection{The positive semidefinite case}

A variant of the minimum rank problem which has been well-studied (see, e.g.,   \cite{booth_et_al_2011}, \cite{booth_et_al_2008}, \cite[Section 46.3]{handbook_min_rank_chapter} and \cite{van_der_holst}) is that in which only positive semidefinite matrices are considered.
For convenience in this context, we set up the following notation, analogous to Notation \ref{not:qual_classes}.

\begin{notation}
Let $G$ be a graph.  We write $\Hqual_+(G)$ for the set of all  positive semidefinite Hermitian matrices with graph $G$.
\end{notation}

\begin{defn}\label{def:msr}
Let $G$ be a graph on $n$ vertices.  The {\it minimum semidefinite rank} of $G$ is
\[ \msr(G) = \min\{ \rank A : A \in \Hqual_+(G)  \}. \]
\end{defn}

We wish to consider the effect of restricting our attention to circulants in the positive semidefinite setting as well.  Hence, we define the following graph parameter.

\begin{defn}\label{def:mscr}
Let $G$ be a circulant graph.  The {\it minimum semidefinite circulant rank} of $G$ is
\[ \mscr(G) = \min\{\rank(A) : A \in \Hqual_+(G) \text{ is a circulant matrix} \}. \]
\end{defn}

The following observation shows that this parameter is well-defined.

\begin{obs}\label{obs:mscrG_well_defined}
Given a circulant graph $G$ on $n$ vertices, let $A$ be its adjacency matrix.  Since $G$ is a circulant, it is $d$-regular for some $d$.
In particular, the graph Laplacian matrix of $G$, namely $dI-A$, is well-known to be positive semidefinite of rank at most $n-1$, and is clearly a circulant matrix.  Hence, $\mscr(G)$ is defined and $\mscr(G) \le n-1$.
\end{obs}

\begin{ex}\label{ex:Complement_of_6_cycle_mcr}
Consider the graph $G=\circulant{6}{\{\pm 2,3\}}$, shown in Figure \ref{fig:circulant_graph_examples}.  In particular, note that $G$ is the complement of the $6$-cycle.  The real symmetric matrix
\[
A = \left[\begin{array}{rrrrrr}
1 & 0 & -2 & 3 & -2 & 0 \\
0 & 1 & 0 & -2 & 3 & -2 \\
\!\!-2 & 0 & 1 & 0 & -2 & 3 \\
3 & -2 & 0 & 1 & 0 & -2 \\
\!\!-2 & 3 & -2 & 0 & 1 & 0 \\
0 & -2 & 3 & -2 & 0 & 1
\end{array}\right]
\]
has graph $G$ and is a circulant matrix with rank $3$.  Meanwhile, from \cite[Theorem 14]{barrett2004} we have that $\mr(G) \ge 3$.  Hence, for this graph we have
\[
\mr(G) = \mrREAL(G) = \mcr(G) = \mcrREAL(G) = 3.
\]
Note that $A$ is not positive semidefinite; its nonzero eigenvalues are $6$
(with multiplicity 2) and $-6$.  In fact, it is shown in Example \ref{ex:Complement_of_6_cycle_poly} that $\mscr(G) = 4$.
\end{ex}

The analog of Question \ref{q:mr_equal_mcr_when} is again of interest.  In particular, for which circulant graphs $G$ do $\msr(G)$ and $\mscr(G)$ differ?  Example \ref{ex:Complement_of_6_cycle_poly} gives one such graph, showing that the two notions are in fact distinct.
More broadly, certain inequalities between the various minimum rank parameters are inherent from the sets over which their respective minima are defined.  These inequalities are illustrated in Figure \ref{fig:mr_32_flavors}, which is annotated with references showing separation between pairs of parameters, when available.

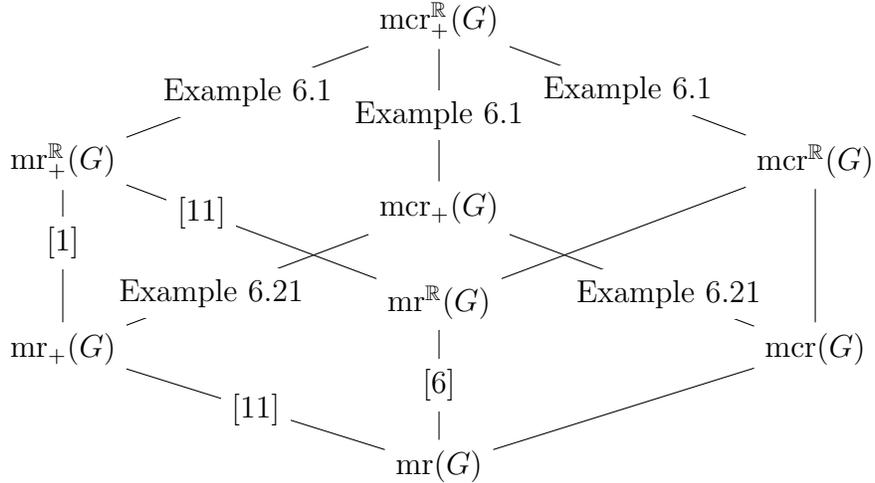
\begin{figure}
\centering
\begin{tikzpicture}[scale=1.25]
	\node (mscr) at (2,3.75) {$\mscr(G)$};
	\node (msr) at (-2,2.25) {$\msr(G)$};
	\node (mcr) at (6,2.25) {$\mcr(G)$};
	\node (mr) at (2,1) {$\mr(G)$};
	\draw (mscr) -- node [fill=white,pos=0.65] {Example \ref{ex:Complement_of_6_cycle_poly}}
	        (msr) -- node[fill=white] {\cite{fallat2007}} (mr);
	\draw (mscr) -- node [fill=white,pos=0.65] {Example \ref{ex:Complement_of_6_cycle_poly}}
	    (mcr) -- (mr);

	\def\x{0};
	\def\y{3};
	\node (mscrREAL) at (2+\x,2.75+\y) {$\mscrREAL(G)$};
	\node (msrREAL) at (-2+\x,1.25+\y) {$\msrREAL(G)$};
	\node (mcrREAL) at (6+\x,1.25+\y) {$\mcrREAL(G)$};
	\node (mrREAL) at (2+\x,-0.25+\y) {$\mrREAL(G)$};

	\draw (mscrREAL) -- node[fill=white] {Example \ref{ex:4cycle_over_R}} (msrREAL);
	\draw (mscrREAL) -- node[fill=white] {Example \ref{ex:4cycle_over_R}} (mcrREAL);
	\draw (mscrREAL) -- node[fill=white] {Example \ref{ex:4cycle_over_R}} (mscr);
	\draw (msrREAL) -- node[fill=white,pos=0.3] {\cite{fallat2007}} (mrREAL);
	\draw (msrREAL) -- node[fill=white,pos=0.4] {\cite{Barioli2010}} (msr);
	\draw (mcrREAL) -- (mrREAL);
	\draw (mcrREAL) -- (mcr);
	\draw (mrREAL) -- node[fill=white] {\cite{Berman2008}} (mr);
\end{tikzpicture}
\caption{Hasse diagram illustrating inequalities that hold by definition among various minimum rank parameters; the values of the parameters are nonincreasing from top to bottom.  Annotations on the edges in the diagram give references showing that strict inequality is possible between the corresponding pairs of parameters.}
\label{fig:mr_32_flavors}
\end{figure}

Although Definition \ref{def:mscr} may seem restrictive,
Theorem \ref{thm:consCircC}
will show that there is at least one natural class of circulant graphs $G$ for which $\mr(G)$ and $\mscr(G)$ coincide, meaning that the smallest rank among all Hermitian matrices with graph $G$ is realized by a positive semidefinite circulant matrix.  Of course, as the following observation notes, all of the minimum rank parameters mentioned here do coincide for the complete graph.

\begin{obs}\label{obs:complete_graph_min_rank}
The $n\times n$ matrix with every entry equal to $1$ witnesses that
\[
\mr(K_n)=\msr(K_n)=\mcr(K_n)=\mscr(K_n)=1.
\]
\end{obs}

\begin{section}{Orthogonal representations and symmetry}
\label{sec:orthogonal_representations_and_symmetry}
A critical tool in studying the minimum semidefinite rank of a graph is the notion of an
{\it orthogonal representation} of the graph.
This is an assignment of a single vector to each vertex of the graph such that the adjacency relation of the graph is captured by the orthogonality relation on the corresponding vectors.  More precisely, we make the following definition.

\begin{defn}
    Let $G$ be a graph and $K$ be a field.  An {\it orthogonal representation} for $G$  in $K^d$ is a function $r:V(G) \rightarrow K^d$ such that, whenever $v,w \in V(G)$ are distinct, $\langle r(v), r(w) \rangle \not=0$ if and only if $\{v,w\} \in E(G)$.
    We also refer to such a function $r$ as an orthogonal representation for $G$ {\it over} $K$.
	The {\it rank} of the orthogonal representation is the dimension of the subspace spanned by $r[V(G)]$.
\end{defn}

Less formally, given a graph $G$ with $n$ vertices, we say that a sequence of $n$ vectors {\it forms an orthogonal representation} for $G$ when some correspondence between the vectors and the vertices of the graph gives an orthogonal representation.

The following theorem gives the significance of the notion of an orthogonal representation in the context of the minimum rank problem.
Although this result is well-known, we include a proof for the sake of completeness, and because it serves as a prototype for the proofs of some analogous results to follow.

\begin{thm}[well-known]
\label{thm:gram_matrix_equivalence}
Let $G$ be a graph.  The following are equivalent.
\begin{enumerate}
\item\label{cond:exists_psd_matrix_w_graph_G_rank_k} There exists a positive semidefinite Hermitian matrix with rank $k$ and graph $G$.
\item\label{cond:exists_rank_k_OR} There exists an orthogonal representation for $G$ over $\mathbb C$ with rank $k$.
\end{enumerate}
\end{thm}

\begin{proof}
Let
$M$ be an $n \times n$ positive semidefinite Hermitian matrix with rank $k$ and graph $G$.
Then $M = A^*A$ for some $k \times n$ matrix $A$ of rank $k$ by  \cite[Theorem 7.2.7]{HJ_2nd_ed}.  Since $M$ has graph $G$, the columns of $A$ form an orthogonal representation for $G$ with rank $k$.

Conversely, let $r$ be an orthogonal representation for $G$ over $\mathbb C$ with rank $k$.
Then the Gram matrix of the vectors in the image of $r$
is a positive semidefinite Hermitian matrix of rank $k$ by \cite[Theorem 7.2.10]{HJ_2nd_ed}.  Since those vectors form an orthogonal representation for $G$, this matrix has graph $G$.
\end{proof}

Note that the 
proof of Theorem \ref{thm:gram_matrix_equivalence} actually
shows that if a graph $G$ has an orthogonal representation over $\mathbb C$ with rank $k$, then $G$ has an orthogonal representation in $\mathbb C^k$.
Combining this observation with
Theorem \ref{thm:gram_matrix_equivalence}
gives the following corollary. 

\begin{cor}\label{cor:msr_is_min_k_for_OR}
Let $G$ be a graph and let $k$ be the smallest integer such that $G$ has an orthogonal representation in $\mathbb C^k$.  Then $k=\msr(G)$.
\end{cor}

Hence, the problem of determining $\msr(G)$ for a particular graph $G$ is equivalent to determining the smallest $k$ such that $G$ has an orthogonal representation in $\mathbb C^k$.  An entirely similar argument shows
that $\msrREAL(G)$ is identical with the smallest $k$ such that $G$ has an orthogonal representation in $\mathbb R^k$.

\begin{ex}\label{ex:Complement_of_6_cycle_OR}
Consider once again the graph $G=\circulant{6}{\{\pm 2,3\}}$ shown in Figure \ref{fig:circulant_graph_examples}.
Let 
\begin{align*}
    v_0 &= \begin{bmatrix} 1 \\ 0 \\ 0 \end{bmatrix}, &
    v_1 &= \begin{bmatrix} 0 \\ 2 \\ 3 \end{bmatrix}, &
    v_2 &= \begin{bmatrix} 1 \\ 3 \\ -2 \end{bmatrix}, &
    v_3 &= \begin{bmatrix} 1 \\ -1 \\ -1 \end{bmatrix}, &
    v_4 &= \begin{bmatrix} 1 \\ 0 \\ 1 \end{bmatrix} &
    \text{and} &&
    v_5 &= \begin{bmatrix} 0 \\ 1 \\ 0 \end{bmatrix}. &
\end{align*}
It is straightforward to verify that $\langle v_i, v_j\rangle \not= 0$ precisely when $i-j \equiv k \bmod 6$ for some $k \in \{\pm 2,3\}$.  Hence, $i \mapsto v_i$ is an orthogonal representation for $G$ of rank $3$,
showing that $\msrREAL(G) \le 3$.
Viewed as extending Example \ref{ex:Complement_of_6_cycle_mcr},
this gives $\msr(G)=\msrREAL(G)=3$.
\end{ex}

The first major goal of the present work is to show that the connection established by Theorem \ref{thm:gram_matrix_equivalence} has a natural analog for the minimum semidefinite circulant rank of Definition \ref{def:mscr}.  In particular, the restriction to circulant matrices corresponds to a requirement that the orthogonal representations considered possess
symmetry in the sense set out precisely as follows.

\begin{defn}\label{def:cyclically_symmetric_OR_in_C}
Let $G$ be a graph with $V(G)=\{0,1,\ldots,n-1\}$.  An orthogonal representation
$r$ for $G$ in $\mathbb C^k$
is said to be {\it cyclically symmetric} if there exists some $k\times k$ unitary matrix $U$ with $U^n=I$ and some $x \in \mathbb C^k$ such that 
$r(i) = U^i x$ for each
$i \in \{0,1,\ldots,n-1\}$.
\end{defn}

The following example serves to illustrate Definition \ref{def:cyclically_symmetric_OR_in_C}, and represents essentially the same construction used by Lov\'asz for \cite[Theorem 2]{lovasz79}.  The discussion of Example \ref{ex:5cycle_poly_example} details how the theory of Section \ref{sec:results_over_the_reals} was applied to construct this example.

\begin{ex}\label{ex:5cycle_cyclically_symm_rep}
Consider the $5$-cycle, $C_5=\circulant{5}{\{\pm 1\}}$.
Figure \ref{fig:depiction_of_C5_orth_rep} shows five vectors in $\mathbb R^3$ that form an orthogonal representation for this graph;
the endpoints of the vectors are at a distance from the $yz$-plane such that $v_i$ and $v_j$ meet at an angle of $\pi/2$ precisely when $i-j \not\equiv \pm 1 \bmod 5$.
Explicitly, these vectors are given by
\[
v_i = \left(
{\textstyle \sqrt{ \scriptstyle 2\cos(\pi/5)\, }},\,
{\textstyle \cos(\frac{2i\pi}5)-\sin(\frac{2i\pi}5)},\,
{\textstyle \cos(\frac{2i\pi}5)+\sin(\frac{2i\pi}5)}
\right)
\]
for $i\in\{0,1,\ldots,4\}$.
In particular, each $v_i$ is the image of $v_{i-1}$ under a rotation about the $x$-axis by an angle of $2\pi/5$, where the subscripts are computed modulo $5$.
Hence, taking $x$ to be any of the five vectors and $U$ to be the unitary matrix whose action on $\mathbb R^3$ is
the aforementioned rotation
shows that Definition \ref{def:cyclically_symmetric_OR_in_C} is satisfied. That is, the orthogonal representation depicted is cyclically symmetric.

\begin{figure}
\centering

\tdplotsetmaincoords{60}{110}
\pgfmathsetmacro{\rvec}{0}
\pgfmathsetmacro{\thetavec}{270}
\pgfmathsetmacro{\phivec}{155}
\begin{tikzpicture}[scale=4.25,tdplot_main_coords]

\tdplotsetcoord{P}{\rvec}{\thetavec}{\phivec}
\tdplotsetthetaplanecoords{\phivec}

\tdplotsetrotatedcoords{\phivec}{\thetavec}{0}

\tdplotsetrotatedcoordsorigin{(P)}

\draw[tdplot_rotated_coords,-latex] (-0.2,0,0) -- (1,0,0) node[above]{$x$};
\draw[tdplot_rotated_coords,-latex] (0,-0.6,0) -- (0,0.7,0) node[left]{$y$};
\draw[tdplot_rotated_coords,-latex] (0,0,-0.7) -- (0,0,0.6) node[left]{$z$};

\draw[color=black,tdplot_rotated_coords,dotted] (0.539344662916632,0.33,0.33) -- (0.539344662916632,-0.214013173973402,0.420024503556700) --
(0.539344662916632,-0.465600748889140,-0.0737439140274914) --
(0.539344662916632,-0.0737439140274915,-0.465600748889140) --
(0.539344662916632,0.420024503556700,-0.214013173973402) --
(0.539344662916632,0.33,0.33);

\draw[-stealth,color=black,tdplot_rotated_coords] (0,0,0) -- (0.539344662916632,0.33,0.33) node[left] {$v_0$};
\draw[-stealth,color=black,tdplot_rotated_coords] (0,0,0) -- (0.539344662916632,-0.214013173973402,0.420024503556700) node[above] {$v_1$};
\draw[-stealth,color=black,tdplot_rotated_coords] (0,0,0) -- (0.539344662916632,-0.465600748889140,-0.0737439140274914) node[right] {$v_2$};
\draw[-stealth,color=black,tdplot_rotated_coords] (0,0,0) -- (0.539344662916632,-0.0737439140274915,-0.465600748889140) node[above right] {$v_3$};
\draw[-stealth,color=black,tdplot_rotated_coords] (0,0,0) -- (0.539344662916632,0.420024503556700,-0.214013173973402) node[above] {$v_4$};
\end{tikzpicture}

\caption{Five vectors forming an orthogonal representation for the $5$-cycle.  That the representation is cyclically symmetric is seen by considering a rotation about the $x$-axis by an angle of $2\pi/5$.}
\label{fig:depiction_of_C5_orth_rep}
\end{figure}
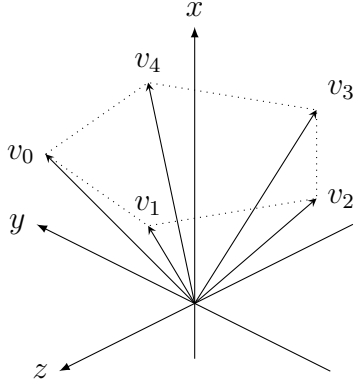
\end{ex}

The following proposition shows that when Definition \ref{def:cyclically_symmetric_OR_in_C} is met by an orthogonal representation for $G$ in $\mathbb R^k$, the associated unitary matrix can be taken to be real orthogonal.

\begin{prop}
Let $G$ be a graph with $V(G)=\{0,1,\ldots,n-1\}$.  If $r:V(G) \rightarrow \mathbb R^k$ is a cyclically symmetric orthogonal representation for $G$, then there exists some $k \times k$ real orthogonal matrix $A$ such that $A^n=I$ and some $x \in \mathbb R^k$
such that $r(i) = A^ix$ for each $i \in \{0,1,\ldots,n-1\}$.
\end{prop}

\begin{proof} Suppose $r:V(G) \rightarrow \mathbb R^k$ is a cyclically symmetric orthogonal representation for $G$.  Then, by Definition \ref{def:cyclically_symmetric_OR_in_C}, there exists some $x \in \mathbb R^k$ and some $k\times k$ unitary matrix $U$ such that $U^n=I$ and $U^ix = r(i) \in \mathbb R^k$ for each $i \in \{0,1,\ldots,n-1\}$.  In particular, letting $W$ be the $\mathbb R$-span of $\{x,Ux,U^2x,\ldots,U^{n-1}x \}$, we have that $W$ is invariant under $U$.  Since $U$ is a unitary transformation, the action of $U$ on this subspace of $\mathbb R^k$ induces a real isometry.  Hence,
there is some real orthogonal matrix $A$ such that $Ay=Uy$ for every $y \in W$, so that, in particular, $r(i)=U^ix=A^ix$ for every $i \in \{0,1,\ldots,n-1\}$.
\end{proof}

A fact that will be important to what follows is that, whenever
a cyclically symmetric orthogonal representation exists, one can be found in a certain very simple canonical form; with a view toward expressing this form, we introduce the following notation.

\begin{notation}\label{not:U_n}
When $n$ is understood to be a positive integer, we write $U_n$ for the diagonal matrix whose $i$th diagonal entry is $\w^i$, i.e., $U_n=\diag{1,\w,\w^2,\ldots,\w^{n-1}}$.  Note that $U_n^n = I$.
\end{notation}

\begin{obs}\label{obs:mult_Un_by_vector_x}
For any integer $i \ge 0$,
\[ U_n^i = \diag{1,\w^i,\w^{2i},\ldots,\w^{(n-1)i}} \]
so that, for any vector $x=(x_0,x_1,\ldots,x_{n-1}) \in \mathbb C^n$,
\begin{equation}\label{eqn:Un_mult_by_vector_x}
U_n^i x = (  x_0,\w^ix_1,\w^{2i}x_2,\ldots,\w^{(n-1)i}x_{n-1} )
\end{equation}
and hence, for any integer $j \ge 0$,
\begin{equation}\label{eqn:UjUj_inner_prod}
\langle U_n^i x, U_n^j x \rangle = \sum_{k=0}^{n-1} x_k\w^{ki}\overline{x_k} \overline{\w^{kj}} = \sum_{k=0}^{n-1} |x_k|^2 \w^{ki}\w^{-kj} = \sum_{k=0}^{n-1} |x_k|^2 \left(\w^{i-j}\right)^k.
\end{equation}
\end{obs}

Unsurprisingly, the matrix $U_n$ has a close connection with the Fourier matrix $F_n$ given in Definition \ref{def:fourier_matrix}.  The following lemma gives one view of this connection.

\begin{lem}\label{lem:spec_decomp_of_Ux_gram_mx}
Given $x=(x_0,x_1,\ldots,x_{n-1}) \in \mathbb R^n$, let $\Lambda=n\diag{x_0^2,x_1^2,\ldots,x_{n-1}^2}$.
Then the Gram matrix of the vectors
$x, U_n x, U_n^2 x, \ldots, U_n^{n-1} x$
is
$F_n^* \Lambda F_n$.
\end{lem}

\begin{proof}
Let $X=\sqrt n\diag{x_0,x_1,\ldots,x_{n-1}}$ and note that $X^2=\Lambda$.
It follows from \eqref{eqn:Un_mult_by_vector_x} that column $i$ of $X F_n$ is given by $U_n^i x$  for each $i \in \{0,1,\ldots,n-1\}$.  Hence, letting $M$ be the Gram matrix of the vectors $x, U_n x, U_n^2 x, \ldots, U_n^{n-1} x$, we have  that $M_{ij}$ is the inner product of the $i$th and $j$th columns of $X F_n$ for every $i,j \in \{0,1,\ldots,n-1\}$.  Thus,
\begin{equation*}
M  = (XF_n)^*(XF_n) = (F_n^*X)(XF_n) = F_n^*X^2F_n = F_n^*\Lambda F_n.
\qedhere
\end{equation*}
\end{proof}

Given any $x \in \mathbb R^n$, the vectors $x, U_n x, U_n^2 x, \ldots, U_n^{n-1} x$ by definition form a cyclically symmetric orthogonal representation for the graph of their
Gram matrix.
Lemma  \ref{lem:spec_decomp_of_Ux_gram_mx} 
shows how the eigenvalues of this matrix arise directly from the coordinates of $x$.
A very useful consequence is that the rank of the Gram matrix, and hence of the associated orthogonal representation, becomes transparent.

\begin{lem}\label{lem:span_of_Ux_vectors_is_size_of_x_support}
If $x \in \mathbb R^n$ has support of size $k$, then the vectors $x,U_nx,U_n^2x,\ldots,U_n^{n-1}x$ span a $k$-dimensional subspace.
\end{lem}

\begin{proof}
By Lemma \ref{lem:spec_decomp_of_Ux_gram_mx}, the Gram matrix of the vectors $x,U_nx,U_n^2x,\ldots,U_n^{n-1}x$ has eigenvalues $nx_0^2,nx_1^2,\ldots,nx_{n-1}^2$, and exactly $k$ of these are nonzero.
\end{proof}

Lemma \ref{lem:span_of_Ux_vectors_is_size_of_x_support} will be especially useful once we have shown that cyclically symmetric orthogonal representations by vectors of the form $x, U_n x, U_n^2 x, \ldots, U_n^{n-1} x$ are essentially the only ones which must be considered.  This fact is anticipated by the following result, which uses Theorem \ref{thm:circulants_diagonalized_by_Fourier} to give Lemma \ref{lem:spec_decomp_of_Ux_gram_mx} a natural interpretation in terms of circulant matrices.

\begin{prop}\label{prop:psd_circulant_iff_Gram_mx_of_Unxs}
An $n\times n$ positive semidefinite Hermitian matrix is a circulant matrix if and only if it is the Gram matrix of the vectors $x,U_nx,U_n^2x,\ldots,U_n^{n-1}x$ for some nonnegative $x \in \mathbb R^n$.
\end{prop}

\begin{proof}
Let $M$ be an $n\times n$ positive semidefinite Hermitian matrix.
Suppose first that $M$ is a circulant matrix.  Then, by Theorem \ref{thm:circulants_diagonalized_by_Fourier}, 
\[
F_nMF_n^* = \Lambda=\diag{\lambda_0,\lambda_1,\ldots,\lambda_{n-1}},
\]
where $\lambda_0,\lambda_1,\ldots,\lambda_{n-1}$ are the eigenvalues of $M$.  Since these are nonnegative,  we may take $x=(\sqrt{\lambda_0/n},\sqrt{\lambda_1/n},\ldots,\sqrt{\lambda_{n-1}/n}) \in \mathbb R^n$.  Then $x$ is nonnegative and, by Lemma \ref{lem:spec_decomp_of_Ux_gram_mx}, $M$ is the Gram matrix of the vectors
$x, U_n x, U_n^2 x, \ldots, U_n^{n-1} x$.

Now suppose $M$ is the Gram matrix of the vectors $x, U_n x, U_n^2 x, \ldots, U_n^{n-1} x$ for some nonnegative $x \in \mathbb R^n$.  Then $M$ is positive semidefinite by \cite[Theorem 7.2.10]{HJ_2nd_ed}.  It remains to show only that $M$ is a circulant matrix, and this follows from the fact that, since $U_n^n=I$,
\[ M_{ij} = \langle U_n^i x, U_n^j x \rangle = \langle U_n^{i-j} x, x \rangle = \langle U_n^{i'-j'} x, x \rangle = \langle U_n^{i'}\!x, U_n^{j'}\!x \rangle = M_{i'j'} \]
for every $i,j,i',j' \in \{0,1,\ldots,n-1\}$ with $i-j \equiv i'-j' \bmod n$.
\end{proof}

The following theorem serves to establish the canonical form for a cyclically symmetric orthogonal representation that was alluded to earlier.  However, its primary significance is seen by analogy with Theorem \ref{thm:gram_matrix_equivalence}.  Just as that theorem showed the existence of a positive semidefinite matrix with a given graph to be equivalent to the existence of an orthogonal representation for that graph with the same rank, the following result shows that a positive semidefinite {\it circulant} matrix with a given graph exists precisely when there is a corresponding {\it cyclically symmetric} orthogonal representation for that graph with the same rank.

\begin{thm}\label{thm:msr_circ_symmetric_rep_equivalence}
Let $G = \circulant{n}{S}$.  The following are equivalent.
    \begin{enumerate}
        \item\label{cond:exists_rank_k_matrix_in_HplusG} There exists a positive semidefinite Hermitian circulant matrix with rank $k$ and graph $G$.
        \item\label{cond:G_has_cyc_symm_rep_in_Ck} There exists a cyclically symmetric orthogonal representation for $G$ over $\mathbb C$ with rank $k$.
        \item\label{cond:G_has_cyc_symm_rep_in_canonical_form} There exists a nonnegative $x \in \mathbb R^n$ with support of size $k$ such that $i \mapsto U_n^i x$ is an orthogonal representation for $G$.
    \end{enumerate}
  \end{thm}

\begin{proof}
The theorem is proved by establishing that \eqref{cond:G_has_cyc_symm_rep_in_Ck} $\implies$
\eqref{cond:exists_rank_k_matrix_in_HplusG} $\implies$ \eqref{cond:G_has_cyc_symm_rep_in_canonical_form} $\implies$ \eqref{cond:G_has_cyc_symm_rep_in_Ck}.

Suppose first that \eqref{cond:G_has_cyc_symm_rep_in_Ck} holds.  Then there exist a vector $x$ and a unitary matrix $V$ with $V^n=I$ such that $i \mapsto V^ix$ gives  an orthogonal representation for $G$ of rank $k$.
Let $M$ be the Gram matrix of the vectors $x,Vx,V^2x,\ldots,V^{n-1}x$.
As in the proof of Theorem \ref{thm:gram_matrix_equivalence}, $M$ is a positive semidefinite Hermitian matrix with rank $k$ and graph $G$.
By an argument identical to that used in the proof of Proposition \ref{prop:psd_circulant_iff_Gram_mx_of_Unxs}, $M$ is also a circulant matrix. 
Hence, \eqref{cond:exists_rank_k_matrix_in_HplusG} holds.

That \eqref{cond:exists_rank_k_matrix_in_HplusG} implies \eqref{cond:G_has_cyc_symm_rep_in_canonical_form} follows directly from Proposition \ref{prop:psd_circulant_iff_Gram_mx_of_Unxs} and Lemma  \ref{lem:span_of_Ux_vectors_is_size_of_x_support}.

Finally, assume that \eqref{cond:G_has_cyc_symm_rep_in_canonical_form} holds, i.e., that there exists some nonnegative $x \in \mathbb R^n$ with support of size $k$ such that $i \mapsto U_n^ix$ gives an orthogonal representation for $G$.  Let $\hat x \in \mathbb R^k$ be the vector formed from the nonzero coordinates of $x$ and let $\hat U$ be the $k\times k$ unitary matrix that results from deleting the $i$th row and column from $U_n$ precisely when $x_i =0$, i.e., let
\[ \hat U = U_n[\{ i : x_i \not= 0 \}]. \]
Using Observation \ref{obs:mult_Un_by_vector_x}, we have that for every $i,j \in \{0,1,\ldots,n-1\}$,
\begin{equation*}
\langle U_n^i x, U_n^j x \rangle = \sum_{k=0}^{n-1} |x_k|^2 \left(\w^{i-j}\right)^k=
\sum_{
x_k \not=\, 0}
|x_k|^2 \left(\w^{i-j}\right)^{k} = \langle \hat U^i \hat x , \hat U^j \hat x \rangle.
\end{equation*}
Hence, $i \mapsto \hat U^i \hat x$ is an orthogonal representation for $G$ in $\mathbb C^k$.  Moreover, this representation has rank $k$ by Lemma \ref{lem:span_of_Ux_vectors_is_size_of_x_support}.  Thus, \eqref{cond:G_has_cyc_symm_rep_in_Ck} holds.
\end{proof}

As noted in Observation \ref{obs:mscrG_well_defined},
whenever $G$ is a circulant graph, there exists a positive semidefinite Hermitian circulant matrix with graph $G$.  By Theorem \ref{thm:msr_circ_symmetric_rep_equivalence}, then, every circulant graph $G$ has a cyclically symmetric orthogonal representation over $\mathbb C$.  It is also clear that no such orthogonal representation may exist unless $G$ is a circulant graph, so that, although Definition \ref{def:cyclically_symmetric_OR_in_C} made no stipulations on the graph, we have the following.

\begin{obs}\label{prop:circ_graph_iff_symm_OR}
A graph $G$ is a circulant graph if and only if there exists a cyclically symmetric orthogonal representation for $G$ over $\mathbb C$.
\end{obs}

Just as the particulars of the proof of Theorem \ref{thm:gram_matrix_equivalence} gave Corollary \ref{cor:msr_is_min_k_for_OR}, the argument that condition \eqref{cond:G_has_cyc_symm_rep_in_canonical_form}
implies condition \eqref{cond:G_has_cyc_symm_rep_in_Ck} in the proof of Theorem \ref{thm:msr_circ_symmetric_rep_equivalence} now gives the following analogous result.

\begin{cor}\label{cor:min_dim_for_symmetric_OR_is_mscr}
    Let $G$ be a circulant graph and let $k$ be the smallest integer such that $G$ has a cyclically symmetric orthogonal representation in $\mathbb C^k$.  Then $k=\mscr(G)$.
\end{cor}

Hence, given a circulant graph $G$, the problem of finding a positive semidefinite matrix with graph $G$ of smallest rank that is a circulant matrix is equivalent to the problem of finding an orthogonal representation for $G$ in the smallest number of dimensions that is cyclically symmetric.

\end{section}

\section{Connection with polynomials}
\label{sec:polynomial_connection}

Let $G$ be a circulant graph on $n$ vertices.  We next show that finding a cyclically symmetric orthogonal representation for $G$ over $\mathbb C$ is equivalent to finding a polynomial satisfying certain combinatorial conditions determined by $G$.  These conditions are entirely in terms of the values of the polynomial on the complex $n$th roots of unity.  We may therefore bound the degree of the polynomials we consider.

\begin{lem}\label{lem:small_degree_poly}
Let $n\in\mathbb{Z}^+$ and $p\in\mathbb{R}[z]$.
Then there exists a unique polynomial $q\in\mathbb{R}[z]$ with $\deg(q) \le n-1$ such that $p(\w^j)=q(\w^j)$ for all $j\in \mathbb Z$.
  \end{lem}

\begin{proof}
Letting $a_0,a_1,a_2,\ldots$ be the sequence of real numbers with $a_i =0$ for $i > \deg(p)$ such that
\begin{align}
	p(z) &= a_0 + a_1z + a_2z^2 + a_3z^3 + \cdots, \notag
\intertext{take}
	q(z) &= (a_0+a_n+\cdots) + (a_1+a_{n+1}+\cdots)z + \cdots + (a_{n-1}+a_{2n-1}+\cdots)z^{n-1}. \label{eqn:q_poly_from_p}
\end{align}
Then, since $(\w^j)^{kn+i}=(\w^j)^i$ for any $i,j,k\in\mathbb{Z}$, the desired property holds.  Uniqueness follows since the values of $q(z)$ are prescribed for $n \ge \deg(q)+1$ distinct values of $z$.
\end{proof}
  
We ultimately wish to show that an appropriate polynomial gives rise to an orthogonal representation for $G$ by vectors in $\mathbb C^n$.  To this end, we establish a correspondence between polynomials and vectors.

\begin{notation}
We write $\mathbb{R}_{\geq 0}[z]$ for the subset of $\mathbb R[z]$ comprising those polynomials with every coefficient nonnegative.
\end{notation}

\begin{defn}\label{defn:normalized_coef_vector}
Let $p \in \mathbb R_{\geq 0}[z]$.
Lemma \ref{lem:small_degree_poly} gives the existence of a unique $q\in\mathbb{R}_{\geq0}[z]$ with degree at most $n-1$ whose value coincides with that of $p$ on every complex $n$th root of unity.  
Say
\[ q(z) = b_0 + b_1z + b_2z^2 + \cdots + b_{n-1}z^{n-1}, \]
with $b_i=0$ for $\deg(q) < i \le n-1$.
Then the {\it normalized coefficient vector} of $p$ is the vector
\[ \left(\sqrt{b_0},\sqrt{b_1},\ldots,\sqrt{b_{n-1}}\right) \in \mathbb R^n.  \]
\end{defn}

Note that the polynomial $q$ referenced in Definition \ref{defn:normalized_coef_vector} is given explicitly by \eqref{eqn:q_poly_from_p}.  Hence, it is straightforward to compute the normalized coefficient vector of any given $p \in \mathbb R_{\geq 0}[z]$.

\begin{defn}\label{defn:poly_corresponding_to_vector}
Given a nonnegative vector $v\in\mathbb{R}^n$, the {\it polynomial corresponding to v} is
\[ v_0^2 + v_1^2z + v_2^2z^2 + \cdots + v_{n-1}^2z^{n-1} \in\mathbb{R}_{\geq0}[z]. \]
\end{defn}

Definitions \ref{defn:normalized_coef_vector} and \ref{defn:poly_corresponding_to_vector} together give a bijective correspondence between polynomials in $\mathbb{R}_{\geq0}[z]$ of degree at most $n-1$ and nonnegative vectors in $\mathbb{R}^n$.
The motivation for setting up this correspondence is made clear by the following lemma.

\begin{lem}\label{lem:innerProduct_with_U_and_p}
Suppose $p \in \mathbb R_{\geq0}[z]$ and $x = (x_0,x_1,\ldots,x_{n-1}) \in \mathbb R^n$ is its normalized coefficient vector.
Then, for every $i \in \{0,1,\ldots,n-1\}$,
\begin{equation*}
 p(\w^i) =  \sum_{j=0}^{n-1} |x_j|^2 \left(\w^i\right)^j = \langle U_n^i x, x \rangle.
\end{equation*}
\end{lem}

\begin{proof}
Since $x$ is the normalized coefficient vector of $p$, the first equality is a direct consequence of Definition \ref{defn:normalized_coef_vector}, while the second follows from 
\eqref{eqn:UjUj_inner_prod}.
\end{proof}

\begin{prop}\label{prop:polynomial_connection}
Suppose $p \in \mathbb R_{\geq0}[z]$ and $x \in \mathbb R^n$ is its normalized coefficient vector.
Then $i\mapsto U_n^ix$ is a cyclically symmetric orthogonal representation for $\circulant{n}{S}$ if and only if $p$ satisfies
    \begin{equation*}
    p(\w^j)=0 \Longleftrightarrow j\not\in S \text{ for all } j \in \{1,2,\ldots,n-1\}.
    \end{equation*}
\end{prop}
  
\begin{proof}
Note that 
$i \mapsto U_n^i x$ is an orthogonal representation for $\circulant{n}{S}$ if and only if
\[ j \not\in S \Longleftrightarrow 0 = \langle x, U_n^j x \rangle \Longleftrightarrow 0 = \langle U_n^j x, x \rangle = p(\w^j) \]
for every $j \in \{1,2,\ldots,n-1\}$, where the final equality holds by Lemma \ref{lem:innerProduct_with_U_and_p}.
\end{proof}

Now Proposition \ref{prop:polynomial_connection} and Theorem \ref{thm:msr_circ_symmetric_rep_equivalence} can be combined to summarize the connection between polynomials and cyclically symmetric orthogonal representations as follows.

\begin{thm}\label{thm:TFAE_poly_C}
Let $G = \circulant{n}{S}$.  The following are equivalent.
  \begin{enumerate}
        \item There exists a positive semidefinite Hermitian circulant matrix with rank $k$ and graph $G$.
        \item There exists a cyclically symmetric orthogonal representation for $G$ over $\mathbb C$ with rank $k$.
    \item There exists a polynomial $p\in\mathbb R_{\geq0}[z]$ whose normalized coefficient vector has support of size $k$ such that
    \begin{equation}\label{eqn:poly_condition}
    p(\w^j)=0 \Longleftrightarrow j\not\in S \text{ for all } j \in \{1,2,\ldots,n-1\}.
    \end{equation}
    \item There exists a polynomial $p\in\mathbb R_{\geq0}[z]$ of degree at most $n-1$ with exactly $k$ terms such that the condition given in \eqref{eqn:poly_condition} holds.
  \end{enumerate}
  \end{thm}

\begin{ex}\label{ex:C4_over_C}
Consider the 4-cycle, $C_4=\circulant{4}{\{\pm 1\}}$.  In the notation of Theorem \ref{thm:TFAE_poly_C}, $n=4$, so that $\w=i$.
The polynomial $p(z)=z+1$ has $p(\w^3)=0$ while $p(\w^1)$ and $p(\w^{-1})$ are nonzero.
Hence, the condition given in \eqref{eqn:poly_condition} is met, so that Theorem \ref{thm:TFAE_poly_C} applies.  Thus, as 
$\deg(p) =1 \le n-1$ and $p$ has $2$ terms, there must exist a cyclically symmetric orthogonal representation for the $4$-cycle of rank $2$, so that $\mscr(C_4) \le 2$.  In particular,
as the normalized coefficient vector of $p$ is $(1,1,0,0) \in \mathbb R^4$,
the vectors
\[ x=\begin{bmatrix}1\\1\\0\\0\end{bmatrix}, \quad U_4x=\begin{bmatrix}1\\i\\0\\0\end{bmatrix}, \quad U_4^2x=\begin{bmatrix}1\\-1\\0\\0\end{bmatrix}, \quad\text{and}\quad U_4^3x=\begin{bmatrix}1\\-i\\0\\0\end{bmatrix} \]
form a cyclically symmetric orthogonal representation for $C_4$ by Proposition \ref{prop:polynomial_connection}.  Since $x$ has support of size $2$, by Lemma \ref{lem:span_of_Ux_vectors_is_size_of_x_support} the rank of this representation should be $2$, which in fact is visibly the case.
Finally, it is easy to verify that $\mr(C_4)=2$, so that in fact
\[ \mr(C_4)=\msr(C_4)=\mscr(C_4)=2.\]
Hence, the $4$-cycle is an example of a graph for which there is a positive semidefinite circulant matrix achieving the minimum rank, and hence the minimum semidefinite rank as well.
\end{ex}

As Example \ref{ex:C4_over_C} illustrates, Proposition \ref{prop:polynomial_connection} allows the construction of a cyclically symmetric orthogonal representation for a given circulant graph on $n$ vertices in terms of a polynomial $p$ with nonnegative coefficients that vanishes on a corresponding selection of the complex $n$th roots of unity.  
Such a polynomial can be taken with degree at most $n-1$, and then the rank of the resulting representation is simply the number of terms appearing in $p$.

This naturally leads  to an interest in the following question.  Given precisely which complex $n$th roots of unity are zeros of a certain polynomial with nonnegative coefficients, how few terms may appear in that polynomial?
In addressing this question, a useful upper bound is provided by the following lemma, the proof of which
rests on a fundamental result of convex geometry.

\begin{lem}\label{lem:poly_from_caratheodory}
Let $W$ be a self-conjugate set of complex $n$th roots of unity with $1 \not\in W$.  Then there exists a polynomial $p\in\mathbb{R}_{\geq0}[z]$ with $deg(p)\le n-1$ 
and at most $|W|+1$ terms
such that $p(w)=0$ for all $w\in W$.
\end{lem}

\begin{proof}
Assume without loss of generality that $W\not=\emptyset$ and take $\alpha_1,\ldots,\alpha_k \in \mathbb C\setminus \mathbb R$ such that
\[
    W\setminus\{-1\} = \{\alpha_1,\overline{\alpha_1},\ldots,\alpha_k,\overline{\alpha_k}\}.
\]
For each $i \in \mathbb Z$, let
\begin{align*}
        v_i = \left((-1)^i,\Re(\alpha_1^i),\Im(\alpha_1^i),\ldots,\Re(\alpha_k^i),\Im(\alpha_k^i) \right) &\in \mathbb R^{|W|} \\
\intertext{if $-1 \in W$, and otherwise let}
        v_i = \left(\Re(\alpha_1^i),\Im(\alpha_1^i),\ldots,\Re(\alpha_k^i),\Im(\alpha_k^i) \right) &\in \mathbb R^{|W|}.
\end{align*}
Every element of $W$ is a root of $z^{n-1}+z^{n-2}+\cdots+z+1$, and it follows that $\sum_{i=0}^{n-1} v_i = 0$.  In particular, $0 \in \mathbb R^{|W|}$ is in the convex hull of the $v_i$ vectors.  It follows by Carath\'eodory's Theorem (see, e.g., \cite[Theorem 2.3]{barvinok}) that there exist nonnegative $b_0,b_1,\ldots,b_m \in \mathbb R$ such that $\sum_{i=0}^{n-1} b_i v_i = 0$ with $|\{i : b_i \neq 0 \}| \le |W|+1$.  Hence, taking
    \[
    p(z) = b_0 + b_1z + b_2z^2 + \cdots + b_{n-1}z^{n-1}
    \]
gives a polynomial as desired.
\end{proof}

Note that standard proofs of Carath\'eodory's Theorem are both elementary and constructive.  Hence, under the hypotheses of Lemma \ref{lem:poly_from_caratheodory}, the polynomial that is asserted to exist can in fact be calculated in a finite number of steps.

When the goal is to apply Theorem \ref{thm:TFAE_poly_C}, the serious limitation of Lemma \ref{lem:poly_from_caratheodory} is that nothing about the result or its proof provides any guarantee as to which complex $n$th roots of unity are {\em not} zeros of the promised polynomial.  Under some     circumstances, however, this limitation can be overcome, as we will see in the next section.

  \begin{section}{Minimum circulant rank for particular classes of circulants}\label{sec:mcr_for_families}

In general, determination of the minimum semidefinite rank $\msr(G)$ for a particular graph $G$ is a difficult problem; the necessary upper and lower bounds may both be difficult to obtain.

A remarkable upper bound that holds in general was proved by probabilistic methods in \cite{LovaszSaksSchrijver1989,LovaszSaksSchrijver1989CORRECTION} and gives the following connection between $\msr(G)$ and the vertex connectivity of $G$, namely the smallest number of vertices whose deletion from $G$ leaves a disconnected graph, which we denote by $\kappa(G)$.

\begin{thm}[{\cite[Corollary 1.4]{LovaszSaksSchrijver1989}}]\label{thm:connectivity_bound}
For every graph $G$ on $n$ vertices, $\msr(G) \le n - \kappa(G)$.
\end{thm}

In particular, let $G=\circulant{n}{S}$.  Since deleting all neighbors of a single vertex is sufficient to disconnect the graph, certainly $\kappa(G) \le |S|$.  When equality holds, Theorem \ref{thm:connectivity_bound} gives $\msr(G) \le n-|S|$.  Interestingly, results of \cite{vanDoorn1986} show that in fact $\kappa(G)=|S|$ does hold under certain conditions, including both when $G$ is a consecutive circulant (see Definition \ref{def:consecutive_circulant}) and when $n$ is prime. Therefore, we have the following.

\begin{thm}\label{thm:msr_inequality_for_prime_and_consecutive}
Let $G=\circulant{n}{S}$.  If $G$ is a consecutive circulant or $n$ is prime, then $\msr(G) \le n-|S|$.
\end{thm}

Also appearing in \cite{LovaszSaksSchrijver1989} is the so-called Delta Conjecture, attributed to
 Maehara,
which asserts that in fact $\msr(G) \le n-\delta$ for every graph on $n$ vertices with minimum degree $\delta$.  Since $\circulant{n}{S}$ is regular of degree $|S|$, the truth of this conjecture would imply that the bound of Theorem \ref{thm:msr_inequality_for_prime_and_consecutive} in fact holds for {\it all} circulant graphs.

In what follows, we strengthen the result of Theorem \ref{thm:msr_inequality_for_prime_and_consecutive} by showing first that $\msr(G)$ can be replaced with $\mscr(G)$, and then that, moreover, this modification actually gives equality.
  In particular, we prove this under the hypothesis that $G$ is a consecutive circulant in 
Subsection \ref{subsec:consecutive_circulants}, and under the hypothesis that $n$ is prime in Subsection \ref{subsec:mscr_for_prime_order_circulants}.  Moreover, our methods of proof are constructive, so that, when the hypothesis of Theorem \ref{thm:msr_inequality_for_prime_and_consecutive} holds, a matrix achieving the minimum of $\mscr(G)$ can be found in a finite number of steps.

\subsection{Computing $\mscr(G)$ for consecutive circulants}
\label{subsec:consecutive_circulants}

Before we proceed, a remark about lower bounds is also in order.  In particular, 
a combinatorial graph parameter called the {\it zero forcing number} of $G$, introduced in \cite{AIM08} and denoted by $Z(G)$, provides a lower bound for the minimum rank over any field.

\begin{thm}[{\cite[Proposition 2.4]{AIM08}}]\label{thm:ZF_lower_bound}
    Let $G$ be a graph on $n$ vertices and let $K$ be a field.  Then $\mr^K(G) \ge n-Z(G)$.
\end{thm}

For consecutive circulant graphs, the zero forcing number behaves in a predictable way.  In particular, the following is straightforward to show.

\begin{thm}
\label{thm:ZF_easy_for_consecutive}
If $G=\circulant{n}{S}$ is a consecutive circulant, then $Z(G)=|S|$.
\end{thm}

Combining Theorems \ref{thm:ZF_lower_bound} and \ref{thm:ZF_easy_for_consecutive} with Theorem \ref{thm:msr_inequality_for_prime_and_consecutive} gives that
\[
n-|S| \le \mr(G) \le \msr(G) \le n-|S|
\]
for any consecutive circulant graph $G=\circulant{n}{S}$, so that, in particular, equality holds throughout, i.e., $\mr(G)=\msr(G)=n-|S|$.  Theorem \ref{thm:consCircC}, which follows,
strengthens this result by showing that $\mscr(G)=n-|S|$ as well.  Hence, for a consecutive circulant graph, the minimum rank
over $\mathbb C$
can always be achieved by a positive semidefinite circulant matrix.  (Example \ref{ex:4cycle_over_R} will show that this does not hold over $\mathbb R$, however.)

Since our goal is to show the existence of an appropriate circulant matrix, we turn our attention to constructing polynomials of the kind required to invoke Theorem \ref{thm:TFAE_poly_C}.  Since these must be polynomials with nonnegative coefficients, we will find the following result, which we quote verbatim from \cite{BDPW}, to be very useful.

\begin{thm}[{\cite[Theorem 1.1]{BDPW}}]\label{thm:pNonneg}
Let $p$ be a polynomial of degree $N$, $p(0)=1$, with  nonnegative coefficients and zeros $z_1, z_2, \ldots, z_N$.  For $t\geq 0$ write
\[
p_t(z) = \prod_{\substack{1\leq j \leq N\\ \left|\Arg(z_j)\right|>t}} \left(1-z/z_j\right).
\]
Then if $p_t\neq p$, all of the coefficients of $p_t$ are positive.
\end{thm}

Taking $p(z)=z^{n-1}+z^{n-2}+\cdots+z+1$ as the initial polynomial $p$ of Theorem \ref{thm:pNonneg}, we obtain that any monic polynomial $q$ whose zeros are precisely
$\w^{k+1},\w^{k+2},\ldots,\w^{n-k-1}$ for some $k$ with $1 \le k < \lfloor n/2 \rfloor$ will have only positive coefficients.  That observation is the essential content of the following lemma.

\begin{lem}\label{lem:pNonNegC}
Whenever $n$ and $k$ are integers with
$1 \le k < \lfloor n/2 \rfloor$,
the polynomial $p(z)=\prod_{j=k+1}^{n-k-1}(z-\w^j)$ has degree $n-2k-1$ and positive real coefficients.
\end{lem}

\begin{proof}
Note that
\begin{equation}\label{eqn:C_over_p_expression}
p(z) = \frac{\prod_{j=1}^{n-1}(z-\w^j)}{\prod_{j=1}^k (z-\w^j)(z-\w^{n-j})}
= \frac{z^{n-1}+\cdots + z^1+1}{\prod_{j=1}^k (z-\w^j)(z-\overline{\w^j})}.
\end{equation}
Hence, $p(z)$ has positive real coefficients by Theorem \ref{thm:pNonneg}.  Moreover, since the numerator and denominator in the rightmost expression of \eqref{eqn:C_over_p_expression} clearly have degree $n-1$ and $2k$, respectively, the degree of $p(z)$ must be $n-2k-1$.
\end{proof}

Given the polynomial supplied by Lemma  \ref{lem:pNonNegC}, we may now apply Theorem \ref{thm:TFAE_poly_C} to obtain the following result.
  
\begin{thm}\label{thm:consCircC}
If $G=\circulant{n}{S}$ is a consecutive circulant graph, then
\[
\mscr(G)=\msr(G)=\mr(G)=n-|S|.
\]
\end{thm}

\begin{proof}
If $G$ is a complete graph, then $|S|=n-1$,
and
the result follows
by Observation \ref{obs:complete_graph_min_rank}.
Hence, assume $G$ is not a complete graph.
Then
$S=\{\pm 1,\pm 2,\ldots,\pm k\}$ for some
integer $k$ with $1 \le k < \lfloor n/2 \rfloor$.
Consider the polynomial $p(z)=\prod_{j=k+1}^{n-k-1}(z-\w^j)$.  Lemma \ref{lem:pNonNegC} implies that $p$ has positive real coefficients
and $\deg(p) = n-2k-1 \le n-1$.
Therefore, exactly $n-2k$ terms appear in $p$.
Moreover, it is easy to verify that $p$ and $S$  satisfy the condition given in \eqref{eqn:poly_condition}.  Hence, Theorem \ref{thm:TFAE_poly_C} gives the existence of a positive semidefinite Hermitian circulant matrix with rank $n-2k$ and graph $G$.  Thus,
\[
\mscr(G) \le n-2k = n-|S|.
\]
Meanwhile, Theorem \ref{thm:ZF_easy_for_consecutive} gives
\[
n-|S| \le \mr(G) \le \msr(G) \le \mscr(G)
\]
for the corresponding lower bound.
\end{proof}
  
Note that, since the $4$-cycle is a consecutive circulant, Example \ref{ex:C4_over_C} may now be seen as a special case of Theorem \ref{thm:consCircC}.

\subsection{Computing $\mscr(G)$ for circulants of prime order}
\label{subsec:mscr_for_prime_order_circulants}

When the goal is to show that a given circulant graph has a cyclically symmetric orthogonal representation of small rank,
Theorem \ref{thm:TFAE_poly_C} allows this to be accomplished by showing that there is a polynomial in $\mathbb R_{\ge 0}[z]$ with a small number of terms whose zeros intersect the complex roots of unity in a way precisely determined by the edges of the graph.

Lemma \ref{lem:poly_from_caratheodory} allows us to place an upper bound on the number of terms that must appear in a polynomial with nonnegative coefficients that vanishes on any desired set of complex $n$th roots of unity.
Unfortunately, the polynomial provided by Lemma \ref{lem:poly_from_caratheodory} may have ``extra" zeros at other, undesired $n$th roots of unity.
Here we investigate the situation further and resolve this issue for the case in which the graph has a prime number of vertices.

Recall from Definition \ref{def:fourier_matrix} that $\hat F_n$ is the Vandermonde matrix of the complex $n$th roots of unity.  Hence, for a vector $(b_0,b_1,\ldots,b_{n-1})\in \mathbb R^n$, the results of evaluating the polynomial
\[
p(z) = b_0 + b_1z +  \cdots + b_{n-1}z^{n-1}
\]
at every $n$th root of unity are precisely the coordinates of the vector $\hat F_nb$.  That is, $p(\w^i) = (\hat F_nb)_i$ for each $i\in\{0,1,\ldots,n-1\}$.
Finding a polynomial which satisfies the condition given in \eqref{eqn:poly_condition}
is therefore equivalent to finding a nonnegative vector in $\mathbb R^n$ which is orthogonal to the $j$th row of $F_n$ precisely when $j\in \{1,2,\ldots,n-1\}\setminus S$.

The proof of the following lemma demonstrates the value in this point of view.

\begin{lem}\label{lem:polynomial_k_terms_bound_on_omega_zeros}
Let $p$ be prime.  A polynomial in $\mathbb R[z]$ with $k$ terms cannot vanish on more than $k-1$ complex $p$th roots of unity.
\end{lem}

\begin{proof}
Let $\w = e^{2\pi i/p}$.
Suppose $q \in \mathbb R[z]$ has $k$ terms and let
\[
S = \{ 0 \le j \le p-1 : q(\w^j)=0 \}.
\]
Assume for the sake of contradiction that $|S| \ge k$.

By Lemma \ref{lem:small_degree_poly}, we may assume without loss of generality that $q$ has degree at most $p-1$.
Then there exists $(b_0,b_1,\ldots,b_{p-1}) \in \mathbb R^p$ with
$q(z) = b_0 + b_1z + \cdots + b_{p-1}z^{p-1}$.
Let $y=\hat F_pb$, so that $q(\w^i)=y_i$ for each $i\in\{0,1,\ldots,p-1\}$,
let $\hat S$ be any subset of $S$ with $|\hat S| = k$,
and let
$A$ be the $k\times k$ matrix given by $A=\hat F_p[\hat S, \{ i : b_i\not=0 \}]$.

Let $\hat b$ be the vector in $\mathbb{R}^k$  obtained from $b$ by deleting all zero coordinates.  Then $A\hat b=0$ and so $A$ is singular.  But this contradicts the fact \cite[Theorem 4]{Delvaux_Barel2008} that every square submatrix of $F_p$ is nonsingular.
\end{proof}

When the result on polynomials provided by Lemma \ref{lem:polynomial_k_terms_bound_on_omega_zeros} is translated via Theorem \ref{thm:TFAE_poly_C} to give information about cyclically symmetric orthogonal representations, it provides the key to both the lower and upper bounds necessary to establish the following result.

\begin{thm}
If $p$ is prime and $G=\circulant{p}{S}$, then $\mscr(G) = p-|S|$.
\end{thm}

\begin{proof}
Let $\w = e^{2\pi i/p}$ and
let $W=\{ \w^j : j\in \{1,2,\ldots,p-1\}\setminus S \}$.  Note that $W$ is self-conjugate.  Therefore,  Lemma \ref{lem:poly_from_caratheodory} gives a polynomial $q$ with $\deg(q) \le p-1$ and at most $|W|+1 = (p-1-|S|)+1 = p-|S|$ terms such that
$q(w)=0$ for every $w \in W$.
Moreover, Lemma \ref{lem:polynomial_k_terms_bound_on_omega_zeros} gives  $q(\w^j)\not=0$ for every $j \in S$.  
Hence, the
condition given in \eqref{eqn:poly_condition} is satisfied, so that by Theorem \ref{thm:TFAE_poly_C} there exists a positive semidefinite Hermitian circulant matrix with rank $p-|S|$ and graph $G$.  Thus, $\mscr(G) \le p-|S|$.

For the reverse inequality, suppose there exists some positive semidefinite Hermitian circulant matrix with rank $k$ and graph $G$.  Then, by Theorem  \ref{thm:TFAE_poly_C}, there is a polynomial $q \in \mathbb R_{\ge 0}[z]$ 
of degree at most $p-1$ with exactly $k$ terms such that the condition given in \eqref{eqn:poly_condition} holds, i.e.,
\[
q(\w^j)=0 \Longleftrightarrow j\not\in S \text{ for all } j \in \{1,2,\ldots,p-1\}.
\]
In particular, $q$ is zero on at least $p-1-|S|$ complex $p$th roots of unity.
It follows by Lemma \ref{lem:polynomial_k_terms_bound_on_omega_zeros} that $k \geq (p-1-|S|)+1 = p-|S|$.
  \end{proof}

  \end{section}

\section{Minimum circulant rank over $\mathbb{R}$}
\label{sec:results_over_the_reals}

In this section we investigate the parameter defined for a circulant graph $G$ as the smallest rank over all {\it real symmetric} positive semidefinite circulant matrices with graph $G$.
Of course, the inequality $\mscr(G) \le \mscrREAL(G)$ trivially holds.
From \cite{van_der_holst} and Example \ref{ex:5cycle_cyclically_symm_rep},
we have for the $5$-cycle that
\[
3 \le \msr(C_5) \le \mscr(C_5) \le \mscrREAL(C_5) \le 3,
\]
showing that equality does occur for some circulant graphs.  On the other hand, the following example shows that the $4$-cycle is a circulant for which equality does not hold.

\begin{ex}\label{ex:4cycle_over_R}
Consider the 4-cycle, $C_4=\circulant{4}{\{\pm 1\}}$.  Over $\mathbb{C}$, we have by either Example \ref{ex:C4_over_C} or the more general result of Theorem \ref{thm:consCircC} that
\[
\mr(C_4)=\msr(C_4)=\mscr(C_4)=2.
\]
Furthermore, it is easily verified that the vectors
\[ \begin{bmatrix}1\\0\end{bmatrix}, \quad \begin{bmatrix}1\\1\end{bmatrix}, \quad \begin{bmatrix}0\\1\end{bmatrix}, \quad\text{and}\quad \begin{bmatrix}1\\-1\end{bmatrix} \]
form an orthogonal representation for $C_4$ in $\mathbb R^2$. Hence, we have in addition that
\[
\mrREAL(C_4)=\msrREAL(C_4)=2.
\]
On the other hand, $\mscrREAL(C_4)\neq 2$.  To see this, let $M$ be a real symmetric circulant matrix with graph $C_4$.  Then
\[ M =
\begin{bmatrix}
b & a & 0 & a \\
a & b & a & 0 \\
0 & a & b & a \\
a & 0 & a & b
\end{bmatrix}
\]
for some $a,b \in \mathbb{R}$.  When $a \neq 0$ and $b=0$, $M$ has rank $2$.  It follows that $\mcrREAL(C_4)=2$.  For $M$ to be positive semidefinite, we may assume without loss of generality that $b=1$.  This leaves $M$ with a characteristic polynomial of
\[  	
x^{4} - 4 x^{3} + (-4 \, a^{2} + 6) x^{2} + (8 \, a^{2}
- 4) x - (4 a^2 - 1).
\]
Hence, for $M$ to be singular requires that $4a^2-1=0$, and this gives $8a^2-4\not=0$, so that the rank of the matrix is $3$.  Hence, $\mscrREAL(C_4) = 3$.
\end{ex}

We now begin a development parallel to that of Section \ref{sec:orthogonal_representations_and_symmetry}, except over $\mathbb R$.  One of the central tools in Section \ref{sec:orthogonal_representations_and_symmetry} was the canonical form provided for a cyclically symmetric orthogonal representation by condition \eqref{cond:G_has_cyc_symm_rep_in_canonical_form} of Theorem \ref{thm:msr_circ_symmetric_rep_equivalence}.  With a view toward establishing an analog of this canonical form, we next introduce a real orthogonal matrix that will play the same role as that of the unitary matrix $U_n$ of Notation \ref{not:U_n}.

\begin{notation}
For any $\theta \in \mathbb R$, let $R_\theta$ denote the $2\times 2$ rotation matrix
\[R_\theta=\left[\begin{array}{cc}\cos\theta & -\sin\theta\\ \sin\theta & \phantom{-}\cos\theta\end{array}\right]
= \begin{bmatrix}
\Re(e^{i\theta}) & \Im(e^{-i\theta}) \\
\Im(e^{i\theta}) & \Re(e^{-i\theta})
\end{bmatrix}
.\]
When $n$ is understood to be a positive integer, we write $A_n$ for the $n \times n$ matrix that has all entries zero with the exceptions that $A_{jj}=1$ when $j=0$, that $A_n[\{j,n-j\}] = R_{2j\pi/n}$
for $1\leq j < \frac{n}{2}$, and that $A_{jj} = -1$ for $j=n/2$ when $n$ is even.
For $k\in\mathbb{Z}$, we have $R_\theta^k=R_{k\theta}$, and so
\[ A_n^k[\{j,n-j\}]
= R_{\frac{2j\pi}n}^k
= R_{\frac{2jk\pi}n}
=
\begin{bmatrix}\cos\left({\scriptstyle \frac{2jk\pi}n}\right) & -\sin\left({\scriptstyle \frac{2jk\pi}n}\right)\\ \sin\left({\scriptstyle \frac{2jk\pi}n}\right) & \phantom{-}\cos\left({\scriptstyle \frac{2jk\pi}n}\right)\end{bmatrix}
= \begin{bmatrix}
\Re(\w^{jk}) & \Im(\w^{(n-j)k}) \\
\Im(\w^{jk}) & \Re(\w^{(n-j)k})
\end{bmatrix}.\]
Note that $A_n$ is permutationally similar to a direct sum of $2 \times 2$ block rotation matrices, a single $1\times 1$ block whose entry is $1$ and, when $n$ is even, another $1\times 1$ block whose entry is $-1$.
In particular, since every block of $A_n$ is orthogonal and has eigenvalues that are $n$th roots of unity, $A_n$ is itself orthogonal, and $A_n^n=I$.
\end{notation}

Just as the Gram matrix of the vectors $x,U_nx,U_n^2x,\ldots,U_n^{n-1}x$ played a central role in Section \ref{sec:orthogonal_representations_and_symmetry}, the Gram matrix of the vectors $x,A_nx,A_n^2x,\ldots,A_n^{n-1}x$ will be critical to what follows.  The following lemma gives the relationship between these two matrices.

\begin{lem}\label{lem:gram_mx_Axs_real_part_of_gram_mx_Uxs}
Let $x=(x_0,x_1,\ldots,x_{n-1}) \in \mathbb R^n$
.  Then, for every $i,j\in\{0,1,\ldots,n-1\}$,
\[ \langle A_n^i x, A_n^j x \rangle = \Re\left( \langle U_n^i x, U_n^j x \rangle \right). \]
\end{lem}

\begin{proof}
Since $A_n$ and $U_n$ are both unitary,
$\langle A_n^i x, A_n^j x \rangle = \langle x, A_n^{j-i}x \rangle$
and
$\langle U_n^i x, U_n^j x \rangle = \langle x, U_n^{j-i}x \rangle$
for every $i,j \in \mathbb Z$.  As a result, it suffices to show that
$\langle x,A_n^kx \rangle = \Re\left(\langle x, U_n^k x \rangle\right)$
for every $k \in \mathbb Z$.  Toward this end, note that $A_n^k + (A_n^k)^T = 2\Re(U_n^k)$, and that, since $U_n^k$ is diagonal, $\Re\left(\langle x, U_n^k x \rangle\right) = \langle x, \Re(U_n^k) x \rangle$.  Therefore,
\[ \Re\left(\langle x, U_n^k x \rangle\right) = \langle x, \Re(U_n^k) x \rangle = \langle x, {\textstyle \frac 12}(A_n^k + (A_n^k)^T) x \rangle. \]
Properties of the real inner product give $\langle x, (A_n^k)^T x \rangle = \langle A_n^k x, x \rangle = \langle x, A_n^k x \rangle$, so that
\[ \langle x, {\textstyle \frac 12}(A_n^k + (A_n^k)^T) x \rangle = {\textstyle \frac 12}\langle x, A_n^k x \rangle + {\textstyle \frac 12}\langle x, (A_n^k)^T x \rangle
= \langle x, A_n^k x \rangle, \]
as desired.
\end{proof}

It will also be useful to understand the support of the vectors $x,A_nx,A_n^2x,\ldots,A_n^{n-1}x$.

\begin{obs}\label{obs:def_Wx}
The block structure of the matrix $A_n$ dictates that the $i$th row of $A_n^k$ may have nonzero entries in only columns $i$ and $n-i$.
Therefore, the $i$th coordinate of $A_n^kx$ may be nonzero only if at least one of $i$ or $n-i$ is in $\supp(x)$.
That is, the support of the vector $A_n^kx$ must be a subset of the set
\begin{equation}\label{eqn:block_support_of_x}
W(x) = \{ i : i \in \supp(x) \text{ or } n-i \in \supp(x) \}.
\end{equation}
As a result, the coordinates of $x$ with indices in $W(x)$ are the only ones that may affect the product $A_n^kx$ for any $k \in \mathbb Z$.
\end{obs}

\begin{defn}\label{def:weight}
Given a vector $x \in \mathbb R^n$, 
the {\it weight} of $x$, denoted $\weight(x)$, is defined by $\weight(x) = |W(x)|$, where $W(x)$ is as given in \eqref{eqn:block_support_of_x}.
\end{defn}

\begin{defn}
A vector $x \in \mathbb R^n$ is {\it balanced} if $x_i = x_{n-i}$ for all $i\in\{1,2,\ldots,n-1\}$. 
\end{defn}

\begin{obs}\label{obs:balanced_vector_has_equal_weight_and_support}
If $x \in \mathbb R^n$ is balanced, then $\weight(x) = \left|\supp(x)\right|$.
\end{obs}

Note that $\weight(x)$ is the number of coordinates that contain a nonzero entry in at least one of the vectors $x,A_nx,A_n^2x,\ldots,A_n^{n-1}x$.
Our next goal is to show that the weight of $x$ in fact gives the dimension of the span of those vectors.  We first need two intermediate results.

\begin{lem}\label{lem:gram_mx_Un_hatxs_gives_transpose_of_gram_mx_Unxs}
Let $x=(x_0,x_1,\ldots,x_{n-1}) \in \mathbb R^n$ and let $\hat x = (x_0,x_{n-1},x_{n-2},\ldots,x_1)$.  Then $\langle U_n^i\hat x, U_n^j\hat x \rangle = \overline{\langle U_n^ix, U_n^jx \rangle} = \langle U_n^jx, U_n^ix \rangle$ for every $i,j\in\{0,1,\ldots,n-1\}$.
\end{lem}

\begin{proof}
The second claimed equality is trivial.  For the first, note, using Observation \ref{obs:mult_Un_by_vector_x}, that
\[ \overline{ \langle U_n^ix, U_n^jx \rangle } = \overline{ \sum_{k=0}^{n-1} |x_k|^2 \left(\w^{i-j}\right)^k } = \sum_{k=0}^{n-1} |x_k|^2 \overline{ \left(\w^{i-j}\right)^k } =  |x_0|^2 + \sum_{k=1}^{n-1} |x_k|^2 \left(\w^{i-j}\right)^{n-k}. \]
Reindexing the rightmost expression then gives
\[ |x_0|^2 + \sum_{k=1}^{n-1} |x_{n-k}|^2 \left(\w^{i-j}\right)^k = \sum_{k=0}^{n-1} |\hat x_k|^2 \left(\w^{i-j}\right)^k = \langle U_n^i\hat x, U_n^j\hat x \rangle. \qedhere \]
\end{proof}

Lemma \ref{lem:spec_decomp_of_Ux_gram_mx} showed the Gram matrix of the vectors $x,U_nx,U_n^2x,\ldots,U_n^{n-1}x$ to have eigenvalues given by $n$ times the squares of the coordinates of $x$.  Hence, one consequence of Lemma \ref{lem:gram_mx_Un_hatxs_gives_transpose_of_gram_mx_Unxs} is that the conjugate
 (not the Hermitian) of this matrix has the same spectrum, 
but with the eigenvalues occurring in a different order in the diagonalized form
that results from conjugation by the Fourier matrix.

Lemma \ref{lem:gram_mx_Axs_real_part_of_gram_mx_Uxs} gave the explicit relationship between the Gram matrix of the vectors \linebreak $x,U_nx,U_n^2x,\ldots,U_n^{n-1}x$ and that of the vectors $x,A_nx,A_n^2x,\ldots,A_n^{n-1}x$, namely that taking the real part of each entry in the first matrix gives the second matrix.  Now, just as Lemma \ref{lem:spec_decomp_of_Ux_gram_mx} gave the spectral decomposition of the first matrix, we are ready to obtain this information for the second matrix as well.

\begin{lem}\label{lem:spec_decomp_of_Ax_gram_mx}
Given $x=(x_0,x_1,\ldots,x_{n-1}) \in \mathbb R^n$, let
\[
\Lambda = n\diag{
x_0^2,
~{\textstyle \frac 12}\left(x_1^2+ x_{n-1}^2\right),
~{\textstyle \frac 12}\left(x_2^2+x_{n-2}^2\right),\ldots,
~{\textstyle \frac 12}\left(x_{n-1}^2+x_1^2\right)}.
\]
Then the Gram matrix of the vectors $x,A_nx,A_n^2x,\ldots,A_n^{n-1}x$ is $F_n^*\Lambda F_n$.
\end{lem}

\begin{proof}
Let $X$ be the Gram matrix of the vectors $x,A_nx,A_n^2x,\ldots,A_n^{n-1}x$ and $Y$ be the Gram matrix of the vectors $x,U_nx,U_n^2x,\ldots,U_n^{n-1}x$.  By Lemma \ref{lem:gram_mx_Axs_real_part_of_gram_mx_Uxs}, then, we have $X = \Re(Y) = (Y + \overline{Y})/2$.
Since each of $Y$ and $\overline Y$ is a circulant matrix, $F_nYF_n^*$ and $F_n\overline YF_n^*$ are both diagonal by Theorem \ref{thm:circulants_diagonalized_by_Fourier}.  In particular,
\[ F_nYF_n^* = n\diag{x_0^2,x_1^2,\ldots,x_{n-1}^2} \]
by Lemma \ref{lem:spec_decomp_of_Ux_gram_mx}, while
\[ F_n\overline Y F_n^* = n\diag{x_0^2,x_{n-1}^2,x_{n-2}^2,\ldots,x_1^2} \]
by the combination of Lemmas \ref{lem:gram_mx_Un_hatxs_gives_transpose_of_gram_mx_Unxs} and \ref{lem:spec_decomp_of_Ux_gram_mx}.  Thus,
    \begin{align*}
        X &= {\textstyle \frac 12 \left(Y + \overline Y\right)} \\
        &= {\textstyle \frac 12 F_n^*n\diag{x_0^2,x_1^2,\ldots,x_{n-1}^2}F_n +  \frac 12 F_n^*n\diag{x_0^2,x_{n-1}^2,x_{n-2}^2,\ldots,x_1^2}F_n} \\
        &=   {\textstyle \frac 12n F_n^* \left(\diag{x_0^2,x_1^2,\ldots,x_{n-1}^2} + \diag{x_0^2,x_{n-1}^2,x_{n-2}^2,\ldots,x_1^2}\right)F_n} \\
        &=   {\textstyle \frac 12n F_n^*\diag{2x_0^2, x_1^2+ x_{n-1}^2,x_2^2+x_{n-2}^2,\ldots,x_{n-1}^2+x_1^2}F_n} \\
        &=   {\textstyle \frac 12n F_n^*(\frac 2n\Lambda) F_n} \\
        &= F_n^*\Lambda F_n. \qedhere
    \end{align*}
\end{proof}

We now obtain as a corollary the following result analogous to Lemma \ref{lem:span_of_Ux_vectors_is_size_of_x_support}.

\begin{lem}\label{lem:span_of_Xx_vectors_is_weight_of_x}
If $x \in \mathbb R^n$ has weight $k$, then the vectors $x,A_nx,A_n^2x,\ldots,A_n^{n-1}x$ span a $k$-dimensional subspace.
\end{lem}

\begin{proof}
The Gram matrix of the vectors $x,A_nx,A_n^2x,\ldots,A_n^{n-1}x$ has eigenvalues
\[
nx_0^2,
~{\textstyle \frac n2}\left(x_1^2+ x_{n-1}^2\right),
~{\textstyle \frac n2}\left(x_2^2+x_{n-2}^2\right),\ldots,
~{\textstyle \frac n2}\left(x_{n-1}^2+x_1^2\right)
\]
by Lemma \ref{lem:spec_decomp_of_Ax_gram_mx}.
Since the entries of $x$ are real, precisely $k$ of these are nonzero.
\end{proof}

The next result is an analog of Proposition \ref{prop:psd_circulant_iff_Gram_mx_of_Unxs}.

\begin{prop}\label{prop:psd_circulant_iff_Gram_mx_of_Anxs_for_balanced_x}
An $n\times n$ real symmetric positive semidefinite matrix is a circulant matrix if and only if it is the Gram matrix of the vectors $x,A_nx,A_n^2x,\ldots,A_n^{n-1}x$ for some balanced nonnegative $x \in \mathbb R^n$.
\end{prop}

\begin{proof}
First, let $M$ be a real positive semidefinite circulant matrix.  By Proposition \ref{prop:psd_circulant_iff_Gram_mx_of_Unxs}, $M$ is the Gram matrix of the vectors $x,U_nx,U_n^2x,\ldots,U_n^{n-1}x$ for some nonnegative $x=(x_0,x_1,\ldots,x_{n-1}) \in \mathbb R^n$.
By Lemma \ref{lem:gram_mx_Axs_real_part_of_gram_mx_Uxs}, since $M$ is real, $M$ is in fact the Gram matrix of the vectors $x,A_nx,A_n^2x\ldots,A_n^{n-1}x$.  Hence, we have $M = F_n^*\Lambda F_n$ for
\begin{align*}
	\Lambda &= \diag{
nx_0^2,
~nx_1^2,
~nx_2^2,
\,\ldots,
~nx_{n-1}^2} \\
	&= \diag{
nx_0^2,
~{\textstyle \frac n2}\left(x_1^2+ x_{n-1}^2\right),
~{\textstyle \frac n2}\left(x_2^2+x_{n-2}^2\right),
\,\ldots,
~{\textstyle \frac n2}\left(x_{n-1}^2+x_1^2\right)},
\end{align*}
where the first equality is from Lemma \ref{lem:spec_decomp_of_Ux_gram_mx} and the second is from Lemma \ref{lem:spec_decomp_of_Ax_gram_mx}.  This shows that, for each $i\in\{1,2,\ldots,n-1\}$,
\[
	{\textstyle \frac n2}\left( x_i^2 + x_{n-i}^2 \right) = nx_i^2,
\]
so that, since $x$ is nonnegative, $x_i=x_{n-i}$.  That is, $x$ is balanced.

For the converse, let $M$ be the Gram matrix of the vectors $x,A_nx,A_n^2x,\ldots,A_n^{n-1}x$ for some balanced nonnegative $x \in \mathbb R^n$.  Trivially, $M$ is real symmetric positive semidefinite.  Since $M$ is real, it follows from Lemma \ref{lem:gram_mx_Axs_real_part_of_gram_mx_Uxs} that $M$ is actually the Gram matrix of the vectors $x,U_nx,U_n^2x,\ldots,U_n^{n-1}x$.  By Proposition \ref{prop:psd_circulant_iff_Gram_mx_of_Unxs}, then, $M$ is a circulant matrix, as desired.
\end{proof}

We are now ready to prove an analog of Theorem \ref{thm:msr_circ_symmetric_rep_equivalence}, the central result of Section \ref{sec:orthogonal_representations_and_symmetry}.
There are actually two possibilities for an analog of condition \eqref{cond:G_has_cyc_symm_rep_in_canonical_form} from that theorem.

\begin{thm}\label{thm:msr_REAL_circ_symmetric_rep_equivalence}
Let $G = \circulant{n}{S}$.  The following are equivalent.
    \begin{enumerate}
        \item\label{cond: psdrcm} There exists a real positive semidefinite circulant matrix with rank $k$ and graph $G$.
        \item\label{cond: cyc sym OR} There exists a cyclically symmetric orthogonal representation for $G$ over $\mathbb R$ with rank $k$.
        \item\label{cond: weight k vector} There exists a nonnegative $x \in \mathbb R^n$ with weight $k$ such that $i \mapsto A_n^i x$ is an orthogonal representation for $G$.
        \item\label{cond: balanced vector} There exists a balanced nonnegative $x \in \mathbb R^n$ with support of size $k$ such that $i \mapsto A_n^i x$ is an orthogonal representation for $G$.
    \end{enumerate}
  \end{thm}

\begin{proof}
We proceed by showing $\eqref{cond: cyc sym OR}\implies \eqref{cond: psdrcm} \implies \eqref{cond: balanced vector} \implies \eqref{cond: weight k vector}\implies\eqref{cond: cyc sym OR}$.  First, $\eqref{cond: cyc sym OR}\implies \eqref{cond: psdrcm}$ follows by an appeal to the rank of the Gram matrix of the vectors of the representation; the argument is identical to that of the analogous implication in the proof of Theorem \ref{thm:msr_circ_symmetric_rep_equivalence}.
  
Now assume that $\eqref{cond: psdrcm}$ holds.  Then there exists a 
real positive semidefinite circulant matrix $M$ with rank $k$ and graph $G$.
By Proposition \ref{prop:psd_circulant_iff_Gram_mx_of_Anxs_for_balanced_x}, $M$ is the Gram matrix of the vectors $x,A_nx,A_n^2x,\ldots,A_n^{n-1}x$ for some balanced nonnegative $x \in \mathbb R^n$.  Thus, $i \mapsto A_n^i x$ is an orthogonal representation for the graph of $M$, which is $G$.  
By Lemma \ref{lem:span_of_Xx_vectors_is_weight_of_x}, we have $\weight(x)=\rank(M)=k$.  Since $x$ is balanced, the support of $x$ has size $k$ by Observation \ref{obs:balanced_vector_has_equal_weight_and_support}.  Hence, \eqref{cond: balanced vector} holds.

That \eqref{cond: balanced vector} implies \eqref{cond: weight k vector} follows immediately from Observation \ref{obs:balanced_vector_has_equal_weight_and_support}.

Finally, assume \eqref{cond: weight k vector} holds.  Then there exists a nonnegative $x \in \mathbb R^n$ with weight $k$ such that $i \mapsto A_n^ix$ is an orthogonal representation for $G$.  Let $\hat x \in \mathbb R^k$ be the vector formed on the coordinates of $x$ with indices in the set $W(x)$ as defined in \eqref{eqn:block_support_of_x} and
let $\hat A$ be the $k\times k$ orthogonal matrix formed on the rows and columns of $A$ with indices in $W(x)$. More precisely, let $\hat A = A_n[W(x)]$.

Note that the support of each block of $A$ that intersects $W(x)$ is entirely contained within $W(x)$, so that by 
Observation \ref{obs:def_Wx} we have $\langle A_n^i x, A_n^j x \rangle = \langle \hat A_n^i \hat x, \hat A_n^j \hat x \rangle$ for every $i,j \in \{0,1,\ldots,n-1\}$.  Therefore, since $i \mapsto A_n^ix$ is an orthogonal representation for $G$, so is $i \mapsto \hat A_n^i\hat x$.  Moreover, since $x$ has weight $k$, it follows from Lemma \ref{lem:span_of_Xx_vectors_is_weight_of_x} that this representation has rank $k$.  Hence, \eqref{cond: cyc sym OR} holds.
\end{proof}

Just as Corollary \ref{cor:min_dim_for_symmetric_OR_is_mscr} followed from Theorem \ref{thm:msr_circ_symmetric_rep_equivalence} and its proof, the next corollary follows in exactly the same way from Theorem \ref{thm:msr_REAL_circ_symmetric_rep_equivalence}.

\begin{cor}\label{cor:min_dim_for_REAL_symmetric_OR_is_mscrREAL}
    Let $G$ be a circulant graph and let $k$ be the smallest integer such that $G$ has a cyclically symmetric orthogonal representation in $\mathbb R^k$.  Then $k=\mscrREAL(G)$.
\end{cor}

Section \ref{sec:polynomial_connection} showed how to connect the equivalent conditions of Theorem \ref{thm:msr_circ_symmetric_rep_equivalence} with related conditions in terms of polynomials.
Our next goal is to develop analogous connections in the real setting.
We first need one additional ancillary result.

\begin{lem}\label{lem:Ux_and_Ax_inner_product_equal_on_balanced_x}
Suppose $x \in \mathbb R^n$ is balanced.  Then $\langle A_n^i x, A_n^j x \rangle = \langle U_n^i x, U_n^j x\rangle$ for every $i,j \in \{0,1,\ldots,n-1\}$.
\end{lem}

\begin{proof}
Choose any $i,j \in \{0,1,\ldots,n-1\}$.
Let $\hat x = (x_0,x_{n-1},x_{n-2},\ldots,x_1)$, but note that actually $\hat x = x$, since $x$ is balanced.
As a result, we have by Lemma \ref{lem:gram_mx_Un_hatxs_gives_transpose_of_gram_mx_Unxs} that
\[
    \langle U_n^i x, U_n^j x \rangle =
    \langle U_n^j \hat x, U_n^i \hat x \rangle =
    \langle U_n^j x, U_n^i x \rangle =
    \overline{\langle U_n^i x, U_n^j x \rangle}.
\]
Thus, $\langle U_n^i x, U_n^j x \rangle$ is real.  Hence, Lemma \ref{lem:gram_mx_Axs_real_part_of_gram_mx_Uxs} gives $\langle A_n^i x, A_n^j x \rangle = \langle U_n^i x, U_n^j x\rangle$.
\end{proof}

In particular,
the Gram matrix of the vectors $x,U_nx,U_n^2x,\ldots,U_n^{n-1}x$ is equal to that of the vectors $x,A_nx,A_n^2x,\ldots,A_n^{n-1}x$ whenever $x \in \mathbb R^n$ is balanced.

In a parallel to Section \ref{sec:polynomial_connection}, we now develop a correspondence between real orthogonal representations and
certain polynomials whose values are constrained at specific complex roots of unity.
Just as each of the two
conditions \eqref{cond: weight k vector} and \eqref{cond: balanced vector} of Theorem \ref{thm:msr_REAL_circ_symmetric_rep_equivalence} can be considered analogous to condition \eqref{cond:G_has_cyc_symm_rep_in_canonical_form} of Theorem \ref{thm:msr_circ_symmetric_rep_equivalence}, we present two results here, each giving a correspondence with polynomials analogous to Proposition \ref{prop:polynomial_connection}.  The first result establishes a polynomial correspondence with condition \eqref{cond: weight k vector} of Theorem \ref{thm:msr_REAL_circ_symmetric_rep_equivalence}.

\begin{prop}\label{prop:real_poly_weight}
Suppose $p \in \mathbb R_{\geq0}[z]$ and $x \in \mathbb R^n$ is its normalized coefficient vector.
Then $i\mapsto A_n^ix$ is a cyclically symmetric orthogonal representation for $\circulant{n}{S}$ if and only if $p$ satisfies
\begin{equation}\label{eqn:poly_condition_R2}
    \Re(p(\w^j))=0 \Longleftrightarrow j\not\in S \text{ for all } j \in \{1,2,\ldots,n-1\}.
\end{equation}
\end{prop}

\begin{proof}
The proof is similar to that of Proposition \ref{prop:polynomial_connection}. Note that
$i \mapsto A_n^i x$ is an orthogonal representation for $\circulant{n}{S}$ if and only if
\[
j \not\in S \Longleftrightarrow 0 = \langle x, A_n^jx \rangle = \langle A_n^jx, x \rangle = \Re(\langle U_n^jx, x \rangle) = \Re(p(\w^j))
\]
for every $j \in \{1,2,\ldots,n-1\}$, where the third and fourth equalities follow from Lemmas \ref{lem:gram_mx_Axs_real_part_of_gram_mx_Uxs} and \ref{lem:innerProduct_with_U_and_p}, respectively.
\end{proof}

The condition given in \eqref{eqn:poly_condition_R2} that the polynomial take on only purely imaginary values on a given set of points may be harder to work with than the condition that the polynomial vanishes on that set.
However, by imposing the additional requirement
that the normalized coefficient vector of the polynomial is balanced, we can obtain a second correspondence, this time in terms of zeros of the polynomial.  In particular, the following result establishes a polynomial correspondence with condition \eqref{cond: balanced vector} of Theorem \ref{thm:msr_REAL_circ_symmetric_rep_equivalence}.

\begin{prop}\label{prop:real_poly}
Suppose $p \in \mathbb R_{\geq0}[z]$ and $x \in \mathbb R^n$ is its normalized coefficient vector.  Suppose also that $x$ is balanced.
Then $i\mapsto A_n^ix$ is a cyclically symmetric orthogonal representation for $\circulant{n}{S}$ if and only if $p$ satisfies
\begin{equation}\label{eqn:poly_condition_R}
    p(\w^j)=0 \Longleftrightarrow j\not\in S \text{ for all } j \in \{1,2,\ldots,n-1\}.
\end{equation}
\end{prop}

\begin{proof}
The proof is identical to that of Proposition \ref{prop:real_poly_weight}, except that now, since $x$ is balanced, we have
$\langle A_n^jx, x \rangle = \langle U_n^jx, x \rangle = p(\w^j)$,
where the first and second equalities follow from Lemmas \ref{lem:Ux_and_Ax_inner_product_equal_on_balanced_x} and \ref{lem:innerProduct_with_U_and_p}, respectively.
\end{proof}

By applying Theorem \ref{thm:msr_REAL_circ_symmetric_rep_equivalence} together with Propositions \ref{prop:real_poly_weight} and \ref{prop:real_poly}, we obtain an analog of Theorem \ref{thm:TFAE_poly_C} that summarizes  the real case as follows.
  \begin{thm}\label{thm:TFAE_poly_R}
Let $G = \circulant{n}{S}$.  The following are equivalent.
    \begin{enumerate}
      \item There exists a real positive semidefinite circulant matrix with rank $k$ and graph $G$.
        \item There exists a cyclically symmetric orthogonal representation for $G$ over $\mathbb R$ with rank $k$.
      \item There exists a polynomial $p\in\mathbb R_{\geq0}[z]$ whose normalized coefficient vector has weight $k$ such that
    \[
    \Re(p(\w^j))=0 \Longleftrightarrow j\not\in S \text{ for all } j \in \{1,2,\ldots,n-1\}.
    \]
      \item \label{cond:balanced_poly_representation}
      There exists a polynomial $p\in\mathbb R_{\geq0}[z]$ whose normalized coefficient vector is balanced and has support of size $k$ such that
    \[
    p(\w^j)=0 \Longleftrightarrow j\not\in S \text{ for all } j \in \{1,2,\ldots,n-1\}.
    \]

    \end{enumerate}
  \end{thm}
  
Now the polynomial used in the proof of Theorem \ref{thm:consCircC} can be modified slightly in order to obtain an analog of that result.

\begin{thm}\label{thm:consCircR}
If $G=\circulant{n}{S}$ is a consecutive circulant graph and $n$ is odd, then
\[
\mscrREAL(G)=\mscr(G)=n-|S|.
\]
\end{thm}

\begin{proof}
As in the proof of Theorem \ref{thm:consCircC}, 
we may assume that $G$ is not a complete graph, so that $S=\{\pm 1, \pm 2, \ldots, \pm k\}$ for some integer $k$ with $1 \le k < \lfloor n/2 \rfloor$.  Also using Theorem \ref{thm:consCircC}, we have $\mscrREAL(G)\geq \mscr(G)=n-|S|=n-2k$.  Hence, we need only show the existence of a real positive semidefinite matrix with graph $G$ and rank $n-2k$.

Let the polynomial
\[p(z)=\prod_{j=k+1}^{n-k-1}(z-\w^j)=\sum_{i=0}^{n-2k-1}a_iz^i\]
and let $q(z)=z^{k+\frac{n+1}2}p(z)$.
By Lemma \ref{lem:pNonNegC}, both $p$ and $q$ have $n-2k$ terms and positive real coefficients.  Moreover, $q$ vanishes at precisely the same complex $n$th roots of unity as does $p$.

Because every zero of $p$ has modulus $1$, its coefficients are symmetric in the sense that $a_i=a_{\deg(p)-i}$ for each $i\in\{0,1,\ldots,\deg(p)\}$.
    As a result, the coefficient of $z^{n+j}$ in $q$ must equal the coefficient of $z^{n-j}$ for each $j\in\{k+1,k+2,\ldots,n-k-1\}$.  It follows that the normalized coefficient vector of $q$ is balanced with weight $n-2k$.  Hence, Theorem \ref{thm:TFAE_poly_R} gives the existence of a matrix as desired.
  \end{proof}

Some of our earlier examples may now be reexamined in terms of polynomials.  For instance,
the orthogonal representation for $C_5$ over $\mathbb R$ given in Example \ref{ex:5cycle_cyclically_symm_rep} was generated by following the proof of Theorem \ref{thm:consCircR}.  We now show the details.

\begin{ex}\label{ex:5cycle_poly_example}
Once again considering the $5$-cycle, this graph satisfies the hypotheses of Theorem \ref{thm:consCircR} with $n=5$ and $S=\{\pm 1\}$.
Hence, as in the proof of that theorem, the polynomial \[p(z)=\prod_{j=2}^{3}(z-\w^j)=z^2+2\cos(\pi/5)z+1\] may be multiplied by $z^4$ to produce the ``shifted'' polynomial
\[
q(z)=z^4p(z)=z^6+2\cos(\pi/5)z^5+z^4,
\] whose normalized coefficient vector is
\[
x=({\textstyle \sqrt{ \scriptstyle 2\cos(\pi/5)\, }},1,0,0,1) \in \mathbb R^5.
\]
Note that $x$ is nonnegative and balanced, with support of size $3$.  Meanwhile, we have by construction that $q(\w^j)=0$ if and only if $j\in S$, so that $q$ satisfies condition \eqref{cond:balanced_poly_representation} of Theorem \ref{thm:TFAE_poly_R}.  
Hence, by Proposition \ref{prop:real_poly},
applying powers of $A_5$ to $x$  must give a cyclically symmetric orthogonal representation for $C_5$.
Stripping the third and fourth coordinates (which are necessarily zero) from the resulting vectors gives the representation shown in Figure \ref{fig:depiction_of_C5_orth_rep}.
  \end{ex}

Theorem \ref{thm:consCircR},
unlike the analogous Theorem \ref{thm:consCircC}, includes the additional hypothesis that $n$ is odd.  Example \ref{ex:4cycle_over_R} shows that
this
is necessary; that is,
the conclusion of Theorem \ref{thm:consCircR} does fail for some consecutive circulants of even order.  
As noted in that example, $C_4=\circulant{4}{\{\pm 1\}}$ is a graph for which the minimum rank
is achieved by a circulant matrix that is Hermitian and positive semidefinite, by a circulant matrix that is real and symmetric but not positive semidefinite, and by a non-circulant matrix that is real and positive semidefinite, but {\it not} by a circulant matrix that is real and positive semidefinite!
In other words,  $\mscr(G)=\mcrREAL(G)=\msrREAL(G)=2$, and yet $\mscrREAL(G)=3$.
We now show how the polynomial results of Theorems \ref{thm:TFAE_poly_C} and \ref{thm:TFAE_poly_R} can be used to give alternate proofs of these equalities, expanding on Examples \ref{ex:C4_over_C} and \ref{ex:4cycle_over_R} and providing explicit orthogonal representations in the process.

\begin{ex}\label{ex:4cycle_summary}
Consider the 4-cycle, $C_4=\circulant{4}{\{\pm 1\}}$.  Here, $n=4$ and $\w=i$.  The matrices $U_4$ and $A_4$ are given by
\[U_4=\begin{bmatrix}
1 & 0 & 0 & 0 \\
0 & i & 0 & 0 \\
0 & 0 & -1 & 0 \\
0 & 0 & 0 & -i
\end{bmatrix} \quad\quad \text{and} \quad\quad A_4=\begin{bmatrix}
1 & 0 & 0 & 0 \\
0 & 0 & 0 & -1 \\
0 & 0 & -1 & 0 \\
0 & 1 & 0 & 0
\end{bmatrix}.\]  It is easy to see that $\mr(C_4)=2$, which will provide a lower bound in what follows.
  
The polynomial $p(z)=z+1$ is zero at $\w^2=-1$ and has a nonzero real part at $\w^1$ and at $\w^{-1}$.  Hence, both of Theorems \ref{thm:TFAE_poly_C} and \ref{thm:TFAE_poly_R} may be applied. The normalized coefficient vector of $p$ is $x=(1,1,0,0) \in \mathbb{R}^4$.  This vector has support of size 2 and weight 3.  Hence, Propositions \ref{prop:polynomial_connection} and \ref{prop:real_poly_weight} give cyclically symmetric orthogonal representations of ranks $2$ and $3$, respectively.  Explicitly, these are the representations formed by
\begin{align*}
  x &= \begin{bmatrix}1\\1\\0\\0\end{bmatrix}, &
  U_4x &= \begin{bmatrix}1\\i\\0\\0\end{bmatrix}, &
  U_4^2x &=\begin{bmatrix}1\\-1\\0\\0\end{bmatrix}, &
  U_4^3x &= \begin{bmatrix}1\\-i\\0\\0\end{bmatrix}, \\
\intertext{and by}
    x&= \begin{bmatrix}1\\1\\0\\0\end{bmatrix},  &
  A_4x &= \begin{bmatrix}1\\0\\0\\1\end{bmatrix},  &
  A_4^2x&=\begin{bmatrix}1\\-1\\0\\0\end{bmatrix},  &
  A_4^3&=\begin{bmatrix}1\\0\\0\\-1\end{bmatrix}.
\end{align*}
That these representations have ranks $2$ and $3$, respectively, is in fact visibly the case.  It follows that
$\mscr(C_4)=2$ and $\mscrREAL(C_4)\leq 3$.
  
For a lower bound corresponding to the latter inequality, we argue as follows.  Suppose a real positive semidefinite circulant matrix with graph $C_4$ and rank $2$ exists.  Then, by Theorem \ref{thm:TFAE_poly_R}, there exists a polynomial $p \in \mathbb R_{\ge 0}[z]$ with $p(\w^2)=p(-1)=0$ whose normalized coefficient vector is balanced with support of size $2$.  We may assume that $\deg(p) \le 3$.  Then either $p$ has positive coefficients on the $z^2$ and constant terms, or positive coefficients on the $z^3$ and $z$ terms.  Discarding a factor of $z$ reduces the second case to the first, so we may take $p(z)=az^2+b$ for some positive $a,b \in \mathbb R$.  But then  $0=p(-1)=a+b$, a contradiction.  Hence, $\mscrREAL(C_4) > 2$.  Therefore, the orthogonal representation over $\mathbb R$ shown above is of minimum rank, and $\mscrREAL(C_4)=3$.
\end{ex}

Theorem \ref{thm:TFAE_poly_R} also enables us to return to the graph $\circulant{6}{\{\pm 2,3\}}$ considered in Examples \ref{ex:Complement_of_6_cycle_mcr} and \ref{ex:Complement_of_6_cycle_OR} and compute its remaining minimum rank parameters.

\begin{ex}\label{ex:Complement_of_6_cycle_poly}
Let $G=\circulant{6}{\{\pm 2,3\}}$, the complement of the 6-cycle
shown in Figure \ref{fig:circulant_graph_examples}.  Now $n=6$ so that $\w=e^{2\pi i/6}$.  Previous examples give $\msrREAL(G)=\mcrREAL(G)=3$.    The following argument shows that $\mscr(G)=\mscrREAL(G)=4$.
    
    First, consider the polynomial $p(z)=\frac{1}{2}z^4+\frac{1}{2}z^3+\frac{1}{2}z^2+1$.  It is easy to check that $p(\w^1)=p(\w^{-1})=0$, while $p(\w^j)\not= 0$ for $j\in\{\pm 2,3\}$.
Moreover, the normalized coefficient vector of $p$ is balanced with support of size $4$.  Hence, from Theorem \ref{thm:TFAE_poly_R}, we have $\mscr(G)\leq\mscrREAL(G)\leq 4$.
    
    Suppose for the sake of contradiction that $\mscr(G)=3$.  Then, by Theorem \ref{thm:TFAE_poly_C}, there is a polynomial $p \in \mathbb R_{\ge 0}[z]$ in which exactly three terms appear with $p(\w^1)=p(\w^{-1})=0$ and $p(\w^j)\not= 0$ for $j\in\{\pm 2,3\}$.  Without loss of generality, take $p(z)=az^\ell+bz^k+1$ with $1 \le k < \ell \le 5$ and $a,b>0$.  Now $p(\w)=a\w^\ell+b\w^k+1=0$, and hence $a\w^\ell+b\w^k$ gives $-1$ as a convex combination of two complex $6$th roots of unity.  Geometrically, it is easy to see that the only possibility is
that $a=b=1$, with $k=2$ and $\ell=4$.  That is, $p(z)=z^4+z^2+1$.  But this gives $p(\w^2)=0$, which is a contradiction.  Hence, $\mscr(G) \neq 3$, and so $\mscr(G)=\mscrREAL(G)=4$.
    \end{ex}

\section{Further directions}

Some natural questions remain unresolved.  For instance, Example \ref{ex:Complement_of_6_cycle_poly} shows that $\msr(G)$ and $\mscr(G)$ may differ.  The general problem of determining exactly when this occurs remains open, however.

\begin{q}
Is it possible to characterize the circulant graphs $G$ for which $\mscr(G)=\msr(G)$?
\end{q}

As noted in Section \ref{sec:intro}, when $G$ has $n$ vertices, the possibility of diagonal dominance implies that the maximum rank of a Hermitian (or real symmetric) matrix with graph $G$ is certainly $n$.  Clearly,
when $G$ is a circulant graph,
a matrix with
rank $n$
can be found that is positive semidefinite, or a circulant matrix, or both simultaneously.
It follows from \cite[Lemma 1.1]{barrett2007} that within the Hermitian matrices in general and also within the positive semidefinite matrices in particular, every rank in between the minimum and $n$ is achievable by some matrix with graph $G$.  Hence, it is natural to ask whether
this is the case
within the circulant matrices as well.

\begin{q}\label{q:circulant_north_lemma}
Given a circulant graph $G$, must there exist a Hermitian circulant matrix with rank $k$ and graph $G$ for every integer $k$ with $\mcr(G) \le k \le n$?  Must there exist such a matrix that is positive semidefinite with rank $k$ for every $k$ with $\mscr(G) \le k \le n$?
\end{q}

Note that Question \ref{q:circulant_north_lemma} has an affirmative answer when $G$ is a consecutive circulant.  In particular, the proof of Theorem \ref{thm:consCircC} used a polynomial from $\mathbb R_{\ge 0}[z]$ of degree $\mscr(G)-1$ in which the maximum possible number, namely $\mscr(G)$, of possible terms appear.  Multiplying this polynomial by successive powers of $(z+2)$ creates a polynomial with each larger number of terms while respecting the condition given in \eqref{eqn:poly_condition} of Theorem \ref{thm:TFAE_poly_C}, thereby inducing a positive semidefinite Hermitian circulant matrix with graph $G$ and rank $k$ for each $k$ with $\mscr(G) \le k \le n$.  (The maximum occurs when the degree of the polynomial reaches $n-1$.)

Of the twelve separations possible in the Hasse diagram of Figure \ref{fig:mr_32_flavors}, references are given that separate all but three of them.  It may interesting to find an explicit graph giving each of the remaining separations, if possible.

\begin{q}
Can an explicit circulant graph $G$ be found such that $\mr(G) \not= \mcr(G)$, $\mcr(G) \not= \mcrREAL(G)$, or $\mrREAL(G) \not= \mcrREAL(G)$?
\end{q}

\section*{Acknowledgments}

The present work resulted from a collaboration of the authors that began at the AMS Spring Central Sectional Meeting in 2013.  We thank Iowa State University for hosting this meeting.  Thanks are also due to David Speyer for suggesting the application of the theorem of Carath\'eodory in the proof of Lemma \ref{lem:poly_from_caratheodory}, and to Leslie Hogben for pointing out the relevant results of \cite{vanDoorn1986}.  The authors also thank the referee for an exceptionally careful reading and many helpful comments.

\bibliographystyle{plain}
\bibliography{minrankcirculants-arxiv}

\end{document}